\documentclass[12pt]{amsart}

\usepackage[margin=1.5in]{geometry}
\usepackage{graphicx}
\usepackage{amsmath}
\usepackage{amsfonts}
\usepackage{mathrsfs} 
 \usepackage{mathptmx} 

\newcommand{\scl}{\mathscr{L}}

\newtheorem{theorem}{Theorem}[section]
\newtheorem{lemma}{Lemma}[section]
\newtheorem{proposition}{Proposition}[section]
\newtheorem{corollary}{Corollary}[section]

\theoremstyle{definition}

\theoremstyle{remark}
\newtheorem{example}{Example}[section]
\newtheorem{remark}{Remark}[section]

\numberwithin{equation}{section}

 \DeclareSymbolFont{largesymbols}{OMX}{yhex}{m}{n}
  \DeclareMathAccent{\widehat}{\mathord}{largesymbols}{"62}
\DeclareMathOperator{\RE}{Re}

\newcommand{\scb}{\mathcal{B}}

\newcommand{\scs}{\mathcal{S}}

\newcommand{\mcm}{M}

\newcommand{\supp}{\text{supp}}

\newcommand{\bS}{\mathbf{S}}
\newcommand{\Bp}{\mathcal{D}'_{\geq 0}(\Omega)}

\title{An extension of H\"{o}rmander's hypoellipticity theorem}
\author[D.P. Herzog]{David P. Herzog}\address{Mathematics Department, Duke University, 
Durham, NC 27708}
\email{dherzog@math.duke.edu}
\email{ntotz@math.duke.edu}
\author[N. Totz]{Nathan Totz}

\begin{document}
\begin{abstract}
Motivated by applications to stochastic differential equations, an extension of H\"{o}rmander's hypoellipticity theorem is proved for second-order degenerate elliptic operators with non-smooth coefficients.  The main results are established using point-wise Bessel kernel estimates and a weighted Sobolev inequality of Stein and Weiss.  Of particular interest is that our results apply to operators with quite general first-order terms.               
\end{abstract}

\maketitle

\section{Introduction}
Let 
\begin{align}
\label{eqn:L}
\scl = f + X_0 + \sum_{j=1}^r Y_j^* Y_j 
\end{align} 
where $X_0, Y_1, Y_2, \ldots, Y_{r}$ are real vector fields on $\mathbb{R}^d$, $f:\mathbb{R}^d\rightarrow \mathbb{R}$ and $^*$ denotes the formal adjoint with respect to the $L^2(\mathbb{R}^d, dx)$ inner product.  The goal of this paper is to study regularity of weak solutions $v$ to the equation  
\begin{align}
\label{eqn:eqn}
\scl v =g 
\end{align}
on a bounded open subset $\Omega\subset \mathbb{R}^d$ where $g: \mathbb{R}^d \rightarrow \mathbb{R}$.  In the broadest sense, we attempt to do this when $f$ and the coefficients of $X_0, Y_1, \ldots, Y_r$ fail to be smooth ($C^\infty)$ functions on $\Omega$, and the operator $\scl$ is degenerate elliptic.

Since research on this topic is vast, it is important at the outset to highlight previous work and note how this paper differs.  The most evident difference from the seminal papers \cite{F91,FL83,SW06} and the references therein is our framework. As our motivations stem from similar partial differential equations arising in the theory of stochastic differential equations: we assume further regularity in $f$ and the coefficients of $X_0, Y_1, \ldots, Y_r$ and we consider non-negative distributional solutions $v$ on $\Omega$ of \eqref{eqn:eqn}.  Therefore, our initial solution space consists of functions which are not weakly differentiable yet we are afforded the luxury of non-negativity of $v$.  Another difference is that we seek a result in this setting that is strikingly reminiscent of H\"{o}rmander's hypoellipticity theorem \cite{Hor67} which can also give (provided the regularity of the coefficients of $\scl$ permits) further regularity of solutions than H\"{o}lder continuity.  In particular, we will see that to the vector fields $X_0, Y_1, \ldots, Y_r$, we may associate a family of smooth vector fields $\mathcal{F}$ on $\mathbb{R}^d$ such that if 
\begin{align*}
\text{Lie}_{x}(\mathcal{F}) = \mathbb{R}^d \,\text{ for every }\,x\in \Omega,
\end{align*}         
then there exists $\bS>0$ (which is a strictly increasing function of the level of regularity in the coefficients of $\scl$) and $\delta >0$ such that for all distributions $v$ of the type described above and all $s< \bS$
\begin{align*}
v, \scl v \in H_{\text{loc}}^{s}(\Omega:\mathbb{R}) \implies v \in H^{s+\delta}_{\text{loc}}(\Omega:\mathbb{R}).   
\end{align*}  
Notice here that the difference between H\"{o}rmander's original result and the one proven here is that the implication above is only valid up to $s<\bS$ whereas H\"{o}rmander's theorem holds for all $s\in \mathbb{R}$.  This should be expected, as regularity of solutions in general cannot greatly exceed the regularity of the coefficients of the differential operator $\scl$ governing them.

The two main strengths of our result are:
\begin{itemize}
\item  We can apply it to situations where $X_0$ is indispensably needed to generate directions in $\text{Lie}(\mathcal{F})$.  To our knowledge, regularity estimates have not been obtained for operators of this form in such generality.   
\item Because we opt to carry out much of the psuedo-differential calculus in physical space (as opposed to Fourier space), sharper estimates are deduced for the parameter $\bS$ above.   
\end{itemize}


Although we choose not to take this approach, the problem of this paper can be studied probabilistically using the Malliavin calculus.  In fact, this stochastic calculus of variations was initiated by Malliavin in \cite{M78} to give a probabilistic proof of H\"{o}rmander's original result.  His program was subsequently carried out in a number of works \cite{Bis81b,Bis81a,KStrI,KStrII,KStrIII,Nor86} and, since then, research has largely centered on extensions in a direction different than the one taken in this paper \cite{Cass09,HM06,Oc88}.  That being said, however, sufficient conditions for the existence and regularity of probability density functions corresponding to stochastic differential equations have been given before \cite{IW84,MSS05} but not in the same light as here.  In future work, it may be interesting to take an in-depth look at the problem from this perspective to see if further insight can be made.

It is also important to point out that there are a plethora of special forms of $\scl$ where our general result does not give optimal regularity of $v$.  For example, if $\scl$ commutes with certain psuedo-differential operators this can induce (see \cite{WangL03}) local smoothness of $v$ in a fixed direction which can then imply regularity in others via the relation $\scl v = g$.  Although we cannot hope to cover all of these cases in the general setting, we layout a framework that can still yield similar results.  In particular, from this paper one can extract results for the parabolic operator $\scl -c \partial_t$, $c\neq 0$ constant, as well.


The structure of this document is as follows.  In Section 2, we fix notation and state the main results.  In the same section, the main results are then applied to a concrete example.  Section 3 outlines the proof of the main results and gives some intuition behind the arguments which establish them.  In Section 4, we derive point-wise Bessel kernel estimates which are then used indispensably in Section 5, Section 6, and Section 7 to establish the essential Sobolev and commutator estimates.

\section{The statement of the main result and an example}

Let us first fix notation.  For $U\subset \mathbb{R}^d$ open and $V\subset \mathbb{C}$, throughout:

\vspace{0.1in}

\noindent $-$  $B(U:V)$ denotes the set of bounded measurable functions $h: U\rightarrow V$.  

\noindent $-$  For $s\in [0,1)$, $C^s(U:V)$ denotes the set of continuous functions $h:U\rightarrow V$ on with H\"{o}lder exponent $s$.

\noindent $-$  For $s\geq 1$, $C^{s}(U: V)$ denotes the space of $\lfloor s \rfloor $-times continuously differentiable functions $h:U\rightarrow V$ whose $\lfloor s\rfloor$th partial derivatives are $(s-\lfloor s\rfloor)$-H\"{o}lder continuous.  

\noindent $-$  $C^{\infty}(U: V)$ denotes the space of infinitely differentiable functions $h:U\rightarrow V$.

\noindent $-$  If $\mathcal{C}(U:V)$ is one of the function spaces above, $\mathcal{C}_0(U:V)$ denotes the set of all $h\in \mathcal{C}(U:V)$ with compact support in $U$.   

\noindent $-$  For $Q$ compact, $C^{\infty}_0(Q: V)$ denotes the space of infinitely differentiable functions $h: \mathbb{R}^d \rightarrow V$ supported in $Q$.

\noindent $-$  $\mathcal{S}(\mathbb{R}^d: \mathbb{C})$ denotes the class of Swartz functions from $\mathbb{R}^d$ into $\mathbb{C}$.

\noindent $-$  $\hat{u}(\xi)= \int_{\mathbb{R}^d} e^{-2\pi i \xi\cdot x} u(x) \, dx$ denotes the Fourier transform of $u$.  

\noindent $-$  We adopt the Einstein summation convention over repeated indices.  For example, $\sum_{l=1}^d Y^l(x) \partial_l$ will be written compactly as $Y^l(x) \partial_l$.  

\noindent $-$  If $Y=Y^{l}(x) \partial_l$ and $\mathcal{C}$ is a function space of the type introduced above, we write $Y\in T \mathcal{C}$ if $Y^l \in \mathcal{C}$ for all $l=1,2, \ldots, d$.

\noindent $-$  For $h: U\rightarrow V$, we will use $h', h'', h''', h^{(4)}, \ldots$ in the context of various norms to indicate that the supremum has been taken over all derivatives of, respectively, order one, two, three, four, $\ldots$.

\noindent $-$  $\| \cdot \|_{p}$, $p\in [1, \infty]$, denotes the $L^p(\mathbb{R}^d, dx)$ norm.

\noindent $-$  $(\cdot, \cdot)$ denotes the $L^2(\mathbb{R}^d, dx)$ inner product.  

\noindent $-$  For $s\in \mathbb{R}$, $H^s(\mathbb{R}^d: \mathbb{R})$ denotes the closure of $C_{0}^\infty(\mathbb{R}^d: \mathbb{R})$ in the norm
\begin{align*}
\| u\|_{(s)} := \sqrt{ \int_{\mathbb{R}^d} |\hat{u}(\xi)|^2 (1+ 4\pi^2 |\xi|^2)^s\, d\xi}
\end{align*}

\noindent $-$  $H^{s}_{\text{loc}}(\Omega: \mathbb{R})$ denotes the set of real distributions $u$ on $\Omega$ such that $\varphi u \in H^{s}(\mathbb{R}^d: \mathbb{R})$ for all $\varphi \in C_0^\infty(\Omega: \mathbb{R})$.  

\noindent $-$  $| \cdot |_{s}$ for $s>0$ denotes the H\"{o}lder norm on the space $C^{s}(\mathbb{R}^d: \mathbb{C})$.

\noindent $-$  For a vector field $Y=Y^l (x) \partial_l$, we write $|Y|_{s}$ as shorthand notation for $\sup_l |Y^l|_{s}$. 

\noindent $-$  $\Bp$ denotes the set of non-negative distributions $v$ on $\Omega$.

\vspace{0.1in}

Because $\scl$ is a linear operator, it suffices to study regularity of $v$ in a small open subset containing the origin.  In particular, we assume throughout that $0\in \Omega $.

Depending on which directions in $\scl$ are needed to span the tangent space, we will employ one of the following two base regularity assumptions on the coefficients of $\scl$:

\textbf{(A1)}\,  $f\in B_0(\mathbb{R}^d: \mathbb{R})$ and $X_0, Y_1, \ldots, Y_r \in TC^1_0(\mathbb{R}^d: \mathbb{R})$.  
 
\textbf{(A2)}\,  $f\in B_0(\mathbb{R}^d: \mathbb{R})$, $X_0\in C^{3/2}_0(\mathbb{R}^d: \mathbb{R})$, and $Y_1, Y_2, \ldots, Y_r \in TC^{5/2}_0(\mathbb{R}^d: \mathbb{R})$.  Moreover if $w_\gamma(x)= |x|^\gamma $,  $X_0= X_0^l(x) \partial_l$ and $Y_j= Y_j^l(x) \partial_l$, then for every $l,j=1,2, \ldots, r$ each of the following is finite:  
\begin{align*}
\|  \,  (X_0^l)'' w_{\gamma_1}\, \|_{\infty}, \, \| (X_0^l)''' w_{\gamma_2}\|_{\infty}, \,  \|  \,  (Y_j^l)''' w_{\gamma_1}\, \|_{\infty}, \, \|  \,  (Y_j^l)^{(4)} w_{\gamma_2} \, \|_{\infty}
\end{align*}   
for some $0\leq \gamma_1 <\min(1,d/2)$,  $0 \leq \gamma_2 <\min(2, d/2)$. 

\begin{remark}
The assumption that the coefficients of $\scl$ are both globally defined and compactly supported is solely for convenience.  
\end{remark}

\begin{remark}
One can replace the weight conditions in \textbf{(A2)} by general Sobolev inequalities (consult Corollary \ref{wcor} and the proof of Theorem \ref{bracketlem} to see how they are used).  These are given as is for concreteness and simplicity.  
\end{remark}

To define the set of smooth vector fields $\mathcal{F}$ in the statement of the main results, let $\mathcal{F}_0$ denote the class of $Y\in TC^\infty_0(\mathbb{R}^d:\mathbb{R})$ satisfying the following comparison condition:  There exists a constant $C>0$ such that 
\begin{align}
\label{eqn:comparison}
\| Y u\|^2 \leq  C \sum_{j=1}^r \|Y_j u \|^2
\end{align}
for all $u \in C_0^{\infty}(\mathbb{R}^d: \mathbb{C})$.  It is important to note here that the constant $C$ is independent of $u\in C_0^{\infty}(\mathbb{R}^d: \mathbb{C}) $.
 $\mathcal{F}_0^B$ denotes the set of $\Upsilon \in TB_0(\mathbb{R}^d: \mathbb{R})$ satisfying the same comparison condition.      
\begin{remark}
The comparison \eqref{eqn:comparison} mimics \emph{subuniticity} as introduced by Fefferman and Phong \cite{FP83}.  The same condition has been used to obtain regularity of weak solutions when $\scl$ is of certain, specific forms. See, for example, \cite{WangL03}.  
\end{remark}

\begin{theorem}
\label{thm:main}
Suppose that for some $\bS\geq 0$: $f\in C^{\bS}_0(\mathbb{R}^d: \mathbb{R})$, $X_0\in TC^{\bS +1}_0(\mathbb{R}^d: \mathbb{R}) $ and $Y_1, \ldots, Y_r \in T C^{\bS+2}_0(\mathbb{R}^d :\mathbb{R})$.  If one of the following two conditions is met:
\begin{itemize} 
\item \textbf{\emph{(A1)}} is satisfied and $\text{\emph{Lie}}_{x}(\mathcal{F}_0)=\mathbb{R}^d$ for every $x\in \Omega$; 
\item \textbf{\emph{(A2)}} is satisfied and $\text{\emph{Lie}}_{x}(\{X\}\cup\mathcal{F}_0)=\mathbb{R}^d$ for every $x\in \Omega$ where $X\in TC_0^{\infty}(\mathbb{R}^d: \mathbb{R})$ is such that $X_0=X+\Upsilon$ for some $\Upsilon \in \mathcal{F}_0^B$;   
\end{itemize}
then there exists $\delta >0$ such that for every $v\in \Bp$ and every $s<\bS$
\begin{align*}
v, \scl v \in H^{s}_{\emph{\text{loc}}}(\Omega: \mathbb{R}) \implies v \in H^{s+\delta}_{\emph{\text{loc}}}(\Omega: \mathbb{R}).  
\end{align*}
\end{theorem}

\begin{remark}
Note that if $f\in C^\infty_0(\mathbb{R}^d: \mathbb{R})$ and $X_0, Y_1, \ldots, Y_r \in TC^\infty_0(\mathbb{R}^d: \mathbb{R})$, Theorem \ref{thm:main} b) retains H\"{o}rmander's result \cite{Hor67} when the weak solution $v$ of \eqref{eqn:eqn} belongs to $\Bp$.  
\end{remark}

\begin{remark}
One could think of the decomposition $X_0=X+\Upsilon$ as the linearization of the coefficients of $X_0$ about the origin.  That is, one could think of $X$ as being a constant vector field and $\Upsilon$ as being a vector field which vanishes at $0$.  Therefore $\Upsilon\in \mathcal{F}_0^B$ provided its coefficients vanish sufficiently fast at 0, the speed of which is determined by the vector fields $Y_1, \ldots, Y_r$ through the comparison condition \eqref{eqn:comparison}.          
\end{remark}

\begin{remark}
In special cases (see \cite{WangL03}), the conclusions of Theorem \ref{thm:main} can be improved if $\scl$ commutes with certain psuedo-differential operators.  In this way, one can extend and even improve the results given here for the parabolic case $\scl -c \partial_t$, $c\neq 0$ constant, since $[\scl -c\partial_t, \partial_t^\alpha ]=0$ where $\alpha\geq 0$ is real.       
\end{remark}

We now give an example to further illustrate the hypotheses and conclusions of Theorem \ref{thm:main}. In the example, to obtain the compactly supported hypotheses on the coefficients, simply modify them appropriately outside of $\Omega$.

\begin{example}

Suppose $d=2$, $\Omega = B_1(0)$, and fix $\gamma >1/4 $ arbitrary.  Let  
\begin{align*}
\scl = \partial_x - (x^2+y^2)^{2(1+\gamma)} \partial_y^2 = X_0 + Y_1^*Y_1
\end{align*}
where $X_0= \partial_x + 4(1+\gamma) y (x^2+y^2)^{1+2\gamma} \partial_y $ and $Y_1= (x^2+y^2)^{1+\gamma}\partial_y$.    Notice we have $Y=(x^2+y^2)^{2k}\partial_y \in \mathcal{F}_0$ for some integer $k>0$ and we may set $X=\partial_x$ in the decomposition $X_0=X+\Upsilon$.  It is easy to check that $\text{Lie}_x(\{ X,Y\})=\mathbb{R}^d$ for all $x\in \Omega$.  Therefore if $v\in \Bp$ satisfies $\scl v =0$, Theorem \ref{thm:main} implies that $v\in H^{s}_{\text{loc}}(\Omega: \mathbb{R})$ for all $s<2\gamma$.  By Sobolev embedding, $v\in C^{t}(\Omega: \mathbb{R})$ for any $t>0$ such that  $1+ t<2\gamma$.  

\end{example}

\section{An outline of the argument}

Here we outline the proof of Theorem \ref{thm:main} and give some intuition for why each part should hold under our hypotheses.

Similar to the arguments in \cite{Hor67,Hor07,Str12}, the proof splits into two parts:

\begin{itemize}
\item  The hypotheses of Theorem \ref{thm:main} imply that $\scl$ is \emph{subelliptic in} $\Omega$; that is, there exists a $\delta >0$ such that for any $Q\subset \Omega$ compact:
\begin{align}
\label{eqn:subelliptic}
\| u\|_{(\delta)} \leq C (\|\mathscr{L} u \|+ \|u\|)
\end{align}       
for all $u \in C_0^{\infty}(Q: \mathbb{C})$.  Here $C>0$ is constant depending on $Q$ but not on $ u \in C_0^{\infty}(Q: \mathbb{C})$.    
\item The hypotheses of Theorem \ref{thm:main} and subellipticity of $\scl$ in $\Omega$ together imply that if $v\in \Bp$ and $s< \bS$
\begin{align*}
v, \scl v \in H_{\text{loc}}^{s}(\Omega: \mathbb{R}) \implies H_{\text{loc}}^{s+\delta}(\Omega: \mathbb{R}).  
\end{align*} 
Here $\delta >0$ is the same constant as above.     
\end{itemize}

To see why subellipticity should be even remotely possible, we now prove the simplest bound giving the local smoothing estimates along the directions contained in $\mathcal{F}_0$.   
\begin{proposition}\label{apest} 
Suppose $Y\in \mathcal{F}_0$ and that \emph{\textbf{(A1)}} is satisfied.  Then there exists a constant $C>0$ such that   
\begin{equation}\label{best}
\|Yu \|^2 + \sum_{k=1}^{r} \| Y_{k} u \|^{2}  \leq C( \text{\emph{Re}}(\mathscr{L} u, u) + \|u\|^{2})   
\end{equation} 
for all $u \in \mathcal{S}(\mathbb{R}^d:\mathbb{C})$.  
\end{proposition}

\begin{remark}
After an application of the Cauchy-Schwarz inequality on $\text{Re}(\scl u, u)$, notice we gain exactly one derivative along any direction contained in the set $\mathcal{F}_0\cup \{ Y_1, \ldots, Y_r\}$.  If these directions generate a basis, then $\scl$ is subelliptic in $\Omega$ with $\delta =1$.  If not, then we must seek more directions by taking iterated commutators of fields in $\mathcal{F}_0$ or $\{ X\} \cup \mathcal{F}_0$.  If a spanning set can be obtained in this way under our hypotheses, it will follow that $\scl$ is subelliptic in $\Omega$ with some small parameter $\delta \in (0,1)$.              
\end{remark}
\begin{proof}[Proof of Proposition \ref{apest}]
Observe that 
\begin{eqnarray}
\label{eqn:energy_eq}
\text{Re}(\mathscr{L} u, u) & =& \sum_{l=1}^{r} \| Y_{k} u \|^{2} + \text{Re}(X_{0} u, u) + (f u,u)\\
\nonumber &=& \sum_{l=1}^{r} \| Y_{k} u \|^{2} + \frac{1}{2}((X_{0} + X_{0}^* +2f)u, u).
\end{eqnarray}
The estimate for $\textstyle{\sum}\|Y_ku\|^2$ now follows since $X_0 + X_0^*  +2f$ is a bounded function on $\mathbb{R}^d$.  By definition of $\mathcal{F}_{0}$, the remainder of \eqref{best} follows immediately.    
\end{proof}

\begin{remark}
\label{rem:8}
If $\text{Lie}_x(\mathcal{F}_0)=\mathbb{R}^d$ for all $x\in \Omega$, then subellipticity of $\scl$ in $\Omega$ follows immediately by the arguments in \cite{Hor07}.  Part of the novelty here is showing how to obtain subellipticity when $\text{Lie}_x(\mathcal{F}_0)$ does not have a basis for some $x\in \Omega$ (see Example 1).  In particular, being able to use $X$ as in Theorem \ref{thm:main} to generate these additional directions while enforcing minimal regularity on the coefficients of $\scl$ is one of our main results.                
\end{remark}

To show the second part of the argument, we will prove the following bound: 
\begin{align}
\label{eqn:boot}
\| \varphi v \|_{(s+\delta)}\leq C ( \| \psi_1 \scl v \|_{(s)} + \|\psi_2 v\|_{(s)})
\end{align} 
for any $v\in \Bp$, $\varphi \in C_0^\infty(\Omega: \mathbb{R})$, and $s< \bS$ where $C>0$ is a constant and $\psi_1, \psi_2 \in C_0^\infty(\Omega: \mathbb{R})$.   Thus if $v, \scl v \in H^s_{\text{loc}}(\Omega: \mathbb{R})$, then the right-hand side above is finite, hence so is the left-hand side giving $v\in H^{s+\delta}_{\text{loc}}(\Omega: \mathbb{R})$.

Intuitively, \eqref{eqn:boot} is obtained by replacing $\delta$ by $s+ \delta$ and the $L^2$ norms by the $H^s$ norms in \eqref{eqn:subelliptic}.  To do this replacement, however, one has to do some non-trivial commuting of operators which is especially difficult under these regularity assumptions.  Moreover, such commuting is only possible for $s< \bS$ because a certain number of derivatives (depending on $s$) must be placed on the coefficients of $\scl$.

\section{Bits of psuedo-differential calculus}\label{sec:frac}  

Here we present mixture of tools from psuedo-differential calculus which are used in subsequent sections to prove the main results.  To be blunt, there is nothing ``new" in this section.  In fact, the contents that follow have been well understood in certain circles for quite some time.  That being said, however, this section is indispensable both as a clean, collected foundation from which the principal result can be established and as a broad picture of how one can bound various psuedo-differential operators with smooth or rough coefficients.

The section is split into two subsections, the first of which covers the usual psuedo-differential calculus of symbols in Fourier space.  We will see that such analysis is most useful in dealing with smooth operators.  The second part, therefore, outlines methods more amenable in the treatment of operators with rough coefficients.  In essence, the main difference between the two programs is that, in the rough setting, most of the calculus is done in physical space (as opposed to Fourier space) via kernel estimates and integration by parts.

Before proceeding on to the individual subsections, we begin by introducing Bessel's operators $\mathcal{B}_s$, $s\in \mathbb{R}$, which play a central role throughout the paper.  Therefore, let $\scb_{s}:\mathcal{S}(\mathbb{R}^d: \mathbb{C})\rightarrow \mathcal{S}(\mathbb{R}^d: \mathbb{C})$, $s \in \mathbb{R}$, be defined by 
\begin{equation}
(\scb_{s}u)(x)= \int_{\mathbb{R}^{d}} e^{2\pi i \xi \cdot x} \hat{u}(\xi) \langle \xi\rangle^{s}\, d\xi
\end{equation} 
where 
\begin{equation}
\label{eqn_symbol}
\langle \xi \rangle^{s}:= (1+4\pi^2 |\xi|^2)^{s/2}.  
\end{equation}
Since derivatives in physical space transform into powers in Fourier space, $\scb_{s}$ simply plays the role of a well-behaved derivative of order $s \in \mathbb{R}$.

Fundamental to utilizing Bessel's operators effectively is the ability to estimate compositions and commutators of operators with $\scb_{s}$ for various values of $s\in \mathbb{R}$.  Depending on the regularity of the coefficients of the operator, to do this one can take one of two paths as now described.

\subsection{Smooth operators}

In the operator in question has smooth coefficients, one can work exclusively in Fourier space by bounding resulting kernels as in the next proposition.         

\begin{proposition}[Schur's test]\label{stest}
Suppose $K:\mathbb{R}^{d}\times \mathbb{R}^{d} \rightarrow \mathbb{C}$ is measurable and satisfies
\begin{equation*}
C:=\max \bigg( \sup_{\xi\in \mathbb{R}^{d}} \int_{\mathbb{R}^{d}} |K(\xi, \eta)| \, d\eta,\, \,  \sup_{\eta \in \mathbb{R}^{d}} \int_{\mathbb{R}^{d}}|K(\xi, \eta) | \, d\xi\bigg)< \infty.  
\end{equation*}
Then the operator $K$ defined by $Ku(\xi) = \int_{\mathbb{R}^{d}} K(\xi, \eta) u(\eta) \, d \eta$ is bounded from $L^{2}(\mathbb{R}^{d}: \mathbb{C})$ into $L^{2}(\mathbb{R}^{d}: \mathbb{C})$.  Moreover, the $L^2$-operator norm of $K$ is precisely $C$.      
\end{proposition}
\begin{proof}
This follows immediately from the Cauchy-Schwarz inequality.  See, for example, Lemma 7.2.4 of \cite{Str12} for a proof.
\end{proof}

To apply Proposition \ref{stest}, one needs to control sums and products of the symbols $\langle\, \cdot \, \rangle^{s}$ in the variables $\xi$ and $\eta$.  This can be done, though not optimally, using the following elementary inequality.

\begin{proposition}[Peetre's Inequality]\label{nineq}
For every $s \in \mathbb{R}$
\begin{equation*}
\frac{\langle \xi \rangle^{s}}{\langle \eta \rangle^{s}} \leq 2^{|s|/2}\langle \xi - \eta \rangle^{|s|} .  
\end{equation*}
\end{proposition}  
\begin{proof}
This is shown by direct computation (cf. Lemma 7.2.5 of \cite{Str12}).  
\end{proof}

We now state and prove a lemma giving the shortest list of estimates we will need when working with smooth operators.  However short the list is, the proofs capture many important elements of the psuedo-differential calculus in Fourier space without requiring the introduction of general symbol classes.

\begin{lemma}\label{pscalc}
Suppose that $F,G\in \scs(\mathbb{R}^{d}:\mathbb{C})$ and $\alpha, \beta \in \mathbb{R}$.  Then there exist constants $C_i >0$ independent of $u\in \scs(\mathbb{R}^d \, : \, \mathbb{C})$ such that   
\begin{align}
\label{pscalc1}\| F u \|_{(\beta)}&\leq   C_1\|u\|_{(\beta)},\\
\label{pscalc2}\|F \partial_k u \|_{(\beta)}&\leq  C_2\|u\|_{(\beta +1)},\\
\label{pscalc3}\|[F \partial_k , \scb_{\alpha} ]u \|_{(\beta)}&\leq  C_3 \|u\|_{(\alpha + \beta)},\\
\label{pscalc4}\| [F\partial_k , \scb_{\alpha} G \partial_j]u \|_{(\beta)} &\leq  C_4\|u \|_{(\alpha + \beta +1)}.
\end{align} 
\end{lemma}

\begin{proof}
We prove the inequalities in order.  Writing $\| F u \|_{(\beta)}= \| \scb_{\beta} Fu\|$, note first that 
\begin{eqnarray*}
\widehat{(\scb_{\beta}F u)}(\xi)= \widehat{(Fu)}(\xi) \langle \xi \rangle^{\beta}&=& \int_{\mathbb{R}^{d}} \hat{F}(\xi-\eta) \hat{u}(\eta) \langle \xi \rangle^{\beta}\, d\eta
\\
&=& \int_{\mathbb{R}^{d}} \hat{F}(\xi-\eta)(\langle \xi \rangle \langle \eta \rangle^{-1})^{\beta} \hat{u}(\eta) \langle \eta \rangle^{\beta}\, d\eta\\
&:=& \int_{\mathbb{R}^{d}}K_{1}(\xi, \eta) \hat{u}(\eta) \langle \eta \rangle^{\beta}\, d\eta.  
\end{eqnarray*}
By Proposition \ref{nineq}, we see that $|K_{1}(\xi, \eta)| \leq 2^{|\beta|/2}|\hat{F}(\xi-\eta)| \langle \xi-\eta \rangle^{|\beta|}$.  Since $\hat{F}(\xi-\eta)$ decays faster than any polynomial in $|\xi-\eta|$ as $|\xi-\eta|\rightarrow \infty$, \eqref{pscalc1} follows by Proposition \ref{stest}.  

For the second inequality, realize by the first inequality that there is a constant $C>0$ independent of $u$ such that  
\begin{eqnarray*}
\| F \partial_k u \|_{(\beta)} \leq C  \|\partial_{k} u\|_{(\beta)}.  
\end{eqnarray*} 
\eqref{pscalc2} now clearly follows from Parseval's identity.  

For the third inequality, note that
\begin{eqnarray*}
\widehat{\scb_{\beta} [F\partial_k , \scb_{\alpha}]u}(\xi)&=&\int_{\mathbb{R}^{d}} \hat{F}(\xi-\eta)\langle \xi \rangle^{\beta}(\langle \eta\rangle^\alpha - \langle \xi\rangle^{\alpha} )\widehat{\partial_{k} u}(\eta) \, d \eta \\
&:=&\int_{\mathbb{R}^{d}} K_{3}(\xi, \eta)\widehat{\partial_{k} u}(\eta)\langle \eta \rangle^{\alpha + \beta-1} \, d \eta 
\end{eqnarray*}
where $K_{3}(\xi, \eta)=\hat{F}(\xi-\eta)\langle \xi \rangle^{\beta}(\langle \eta\rangle^\alpha - \langle \xi\rangle^{\alpha} )\langle \eta\rangle^{-\alpha -\beta+1}$.  By the Fundamental Theorem of Calculus and the Cauchy-Schwarz inequality,
\begin{eqnarray*}
|\langle \eta \rangle^{\alpha}- \langle \xi \rangle^{\alpha}| &=&\bigg| \int_{0}^{1} \frac{d}{dt} \langle t \xi + (1-t) \eta \rangle^{\alpha} \, dt\bigg| \\
&\leq & \alpha  \langle \xi-\eta \rangle \int_{0}^{1}\langle t \xi + (1-t)\eta \rangle^{\alpha -1}\, dt .      
\end{eqnarray*}
Hence by Proposition \ref{nineq}, there exists constants $C, C'>0$ independent of $u$ such that 
\begin{eqnarray*}
|K_{3}(\xi, \eta)| &\leq & C|\hat{F} (\xi- \eta) | \langle \xi-\eta\rangle^{|\beta| +1 }\int_0^{1} \langle t(\eta-\xi) \rangle^{|\alpha-1|} \, dt  \\
&\leq&C' |\hat{F} (\xi- \eta) | \langle \xi-\eta\rangle^{|\beta| + |\alpha-1|+1 }. 
\end{eqnarray*}
As before, since $\hat{F}(\xi-\eta)$ decays faster than any polynomial as $|\xi-\eta |\rightarrow \infty$, we see that by Schur's test
\begin{equation*}
\|[F\partial_k, \scb_{\alpha} ]u \|_{(\beta)}\leq C \| \partial_{k} u\|_{(\alpha + \beta-1)}
\end{equation*}
for some constant $C>0$.  Thus \eqref{pscalc3} now follows from Parseval's identity.    

Finally, to see why the fourth inequality holds, first write
\begin{equation*}
\scb_{\beta}[F\partial_k, \scb_{\alpha} G \partial_j]u= \scb_{\beta} [F\partial_k, \scb_{\alpha}]G \partial_j u + \scb_{\alpha + \beta} [F\partial_k,G\partial_j]u.  
\end{equation*}  
The estimate for the first term on the right follows from an application of \eqref{pscalc3}.  The estimate for the second term follows after applying the second inequality since $[F\partial_k,G\partial_j]=F\partial_k(G)\partial_j - G\partial_j(F) \partial_k$.   
\end{proof}

\begin{remark}
In the preceding arguments, note how the assumption $F, G \in \mathcal{S}(\mathbb{R}^d: \mathbb{R})$ was exploited.  Certainly we do not need its full strength, but even if we were to keep careful track of how much decay in $\hat{F}, \,\hat{G}$ at infinity gives the estimates in the fashion above, the assumptions on $F$ and $G$ produced would not be optimal.  

In certain instances, one can do a similar analysis in Fourier space by, in light of Lemma X1 of \cite{KP88}, either modifying or applying a result of Coifman-Meyer \cite{CM78}.  We, however, found the program in the originating space to be more illuminating and sharp.  
\end{remark}

In light of the previous remark, we turn our attention to:


\subsection{Rough operators}

First notice that $\scb_{-s}$, $s >0$, satisfies the relation (cf. \cite{S70}):     
\begin{equation}\label{brp}
(\scb_{-s}u)(x)= (u * G_{s})(x), \, \, u \in \mathcal{S}(\mathbb{R}^d: \mathbb{C}), 
\end{equation}     
where $*$ denotes convolution and the kernel $G_{s}$ has the integral representation  
\begin{equation*}
G_{s}(x) = \frac{1}{(4 \pi)^{s/2}} \frac{1}{\Gamma(s/2)} \int_{0}^{\infty} e^{-\pi|x|^2/w} e^{-w/4\pi} w^{(-d+s)/2}w^{-1}\, dw, \, \, x\in \mathbb{R}^d_{\neq 0}. 
\end{equation*}
Moreover for $s>0$, the operator $\scb_{-s}$ and its corresponding kernel $G_s$ have the following properties (see \cite{S70} for (p1)-(p4) and \cite{AS61} for the rest):
\begin{itemize}
\item[(p1)] $\scb_{-s}$ extends to a bounded operator from $L^{2}(\mathbb{R}^{d}: \mathbb{C})$ to $L^{2}(\mathbb{R}^{d}: \mathbb{C})$ via $(\scb_{-s} u)(x) =(u*G_{s})(x)$ for $u\in L^{2}(\mathbb{R}^{d}: \mathbb{C})$;
\item[(p2)] $G_{s}\in C^{\infty}(\mathbb{R}^{d}_{\neq 0}:[0, \infty))$ and $G_{s} \in L^{1}(\mathbb{R}^d : [0, \infty))$;
\item[(p3)] $\widehat{G_{s}}(\xi) = \langle \xi \rangle^{-s}$;
\item[(p4)] For $u \in L^{2}(\mathbb{R}^{d}: \mathbb{C})$ and $t >0$, $\scb_{-s}\scb_{-t}u = \scb_{-(s+t)}u$;  
\item[(p5)] 
\begin{equation*}
G_{s}(x)= \frac{1}{2^{\frac{d+s-2}{2}}\pi^{d/2} \Gamma(s/2)} |x|^{\frac{s-d}{2}} K_{\frac{d-s}{2}} (|x|) 
\end{equation*} 
where $K_{v}(z)$ is the modified Bessel function of the third kind of index $v\in \mathbb{R}$;
\item[(p6)]  $K_{v}(z)=K_{-v}(z)$ for all $v\in \mathbb{R}$;
\item[(p7)]  As $z\rightarrow 0$
\begin{equation*}
K_{v}(z) \sim 
\begin{cases}
2^{v-1} \Gamma(v) z^{-v} &\text{ for } v>0\\
\log(1/z) &\text{ for } v=0  
\end{cases}
\end{equation*}
where $\sim$ denotes asymptotic equivalence;
\item[(p8)] As $z\rightarrow \infty$ 
\begin{equation*}
K_{v}(z) \sim \bigg(\frac{\pi}{2 z}\bigg)^{1/2} e^{-z}
\end{equation*}
for all $v \in \mathbb{R}$;
\item[(p9)] For all $v \in \mathbb{R}$ 
\begin{equation*}
\frac{d}{dz} [z^{-v}K_{v}(z)]=-z^{-v} K_{v+1}(z).  
\end{equation*}
\end{itemize}          

Each of (p1)-(p9) is paramount to carrying out the calculus in physical space, in the sense that it allows one to do integration by parts and then bound the resulting quantities efficiently.  The reader does not need memorize each of these properties; the list is simply to be referred to as necessary.  

To afford flexibility later, it is convenient to work more generally with a class of kernels sharing similarities with $G_s$.  Therefore, for $s >0$ let $\mathcal{J}_{s}$ denote the set of functions $J_s:\mathbb{R}^d_{\neq 0} \rightarrow \mathbb{R}$ which are finite linear combinations of functions of the form    
\begin{align*}
x^\sigma D^\tau G_{t}(x)
\end{align*}
where $t >0$, and $\sigma$ and $\tau$ are multi-indices satisfying $|\sigma| -|\tau| +t \geq s$.  By linearity and the product rule, it is clear that if $|\sigma|- |\tau | +s \geq  t>0$ then we have the following closure  
\begin{align*}
x^\sigma  D^{\tau} J_s\in \mathcal{J}_{t}\, \, \text{ whenever }\, \,  J_s\in \mathcal{J}_{s}.  
\end{align*} 
Moreover:

\begin{lemma}\label{lem:kernel_a}
Suppose $J_s \in \mathcal{J}_{s}$ for some $s>0$.  Then: 
\begin{itemize} 
\item[(1)]  $J_s \in C^{\infty}(\mathbb{R}^d_{\neq 0}: \mathbb{R})\cap L^1(\mathbb{R}^d: \mathbb{R})$;  
\item[(2)]  There exists a constant $C>0$ such that 
\begin{align*}
|\widehat{J_s}(\xi)|\leq C \langle \xi\rangle^{-s}, \qquad \forall \xi \in \mathbb{R}^d ; 
\end{align*}
\item[(3)]  For each $t \geq 0$, there exist constants $s_i \geq s + t$ and $C>0$ such that   
\begin{align*}
|x|^{t} |J_s(x)| \leq C \sum_{i=1}^m G_{s_i}(x) ,  \qquad \forall x\in \mathbb{R}^d_{\neq 0}.  
\end{align*}    
\end{itemize}   
\end{lemma}

\begin{proof}
Fix $s >0$ and let $J_s\in \mathcal{J}_s$.  Clearly by definition and (p2), $J_s \in C^{\infty}(\mathbb{R}^d_{\neq 0}: \mathbb{R})$.  To obtain the remainder of the lemma, let $v\geq 0, t >0$ and $|\sigma|=k$ be such that $v+t -k>0$.  We first show that there exists constants $C>0, s_i \geq v + t -k$ such that 
\begin{align}    
\label{est:ptwiseker}
|x|^{v} |D^{\sigma} G_t(x)| \leq C \sum_{i=1}^m G_{s_i}(x), \, \, \forall x\neq 0.
\end{align}
From this, we immediately deduce part (3) of the lemma and see that $J_s \in L^1(\mathbb{R}^d: \mathbb{R})$, thus also finishing the proof of part (1).  To show \eqref{est:ptwiseker}, inductively compute derivatives of $G_t$ using (p5) and (p9) to see that the following estimate holds  
\begin{eqnarray*}
|x|^{v} |D^{\sigma} G_{t}(x)| \leq C  \sum_{\substack{p,p'\in \mathbb{N} \cup \{ 0\}\\p+p'=k\\ p\leq p'}} |x|^{\frac{t-d}{2}+ v-p}|K_{\frac{d-t}{2}+p'}(|x|) |
\end{eqnarray*}
for all $x\in \mathbb{R}^d_{\neq 0}$, for some $C>0$.  Using this estimate, one can then deduce \eqref{est:ptwiseker} by applying the asymptotic formulas contained in (p7) and (p8) case by case.

We have left to verify part (2) of the lemma.  To see this, for $t \in \mathbb{R}$ define 
\begin{align*}
H_t (x) = \int_{0}^{\infty} e^{-\pi|x|^2/s} e^{-s/4\pi} s^{(-d+t)/2}s^{-1}\, ds, \qquad x\in \mathbb{R}^d_{\neq 0}.  
\end{align*}
We will first show by induction on $k\geq 1$ that if $|\sigma|+ |\tau|=k$, then 
\begin{align*}
x^{\sigma} D^\tau H_{t}(x) =\sum_{\upsilon} c_\upsilon D^\upsilon H_{t_\upsilon}(x), \qquad x\neq 0,
\end{align*}
where the sum is finite, $c_\upsilon \in \mathbb{R}$, $t_\upsilon -|\upsilon|\geq |\sigma| -|\tau| +t $.  Consider the case when $k=1$.  Notice that 
\begin{align*}
x^i H_t =- 2\pi \partial_i H_{t+2}.
\end{align*}  
Moreover, any $\partial_i H_{t}$ is already in the prescribed form.  Now suppose that the statement holds for each $1\leq k\leq k'$.  We show that the statement is also valid for $k=k'+1$.  Consider first $x^{\sigma} D^\tau H_t$ with $|\tau|\geq 1$.  Then it follows that
\begin{align*} 
  x^{\sigma} D^\tau H_t = D^\tau (x^{\sigma} H_t) + R  
\end{align*}
where $R$ is in the prescribed form by induction.  Also, we may use the inductive assumption to write 
$$x^\sigma H_t = \sum_{\upsilon} c_\upsilon D^\upsilon H_{t_\upsilon}$$
where $c_\upsilon \in \mathbb{R}$, $t_\upsilon - |\upsilon| \geq |\sigma| + t$.  This now finishes the case when $|\tau |\geq 1$.  If $|\tau|=0$ and $\sigma =(i_1, \ldots, i_k)$, then 
\begin{align*}
x^{i_1} \cdots x^{i_k} H_t = -\tfrac{1}{2\pi}x^{i_1} \cdots x^{i_{k-1}} \partial_{i_k} H_{t+2}.   
\end{align*}  
Now apply the same reasoning as in the case when $|\tau|\geq 1$ to finish the inductive argument.  

Part (2) of the lemma now follows by the inductive argument, standard Fourier analysis and (p3).  
\end{proof}

We conclude the section by proving two basic inequalities needed later.  Below, we use the standard trick of assuming smoothness of the coefficients in question and then keep careful track of how estimates depend on various H\"{o}lder norms of these coefficients.  Later, we will see how we can mollify, take limits, and then control various Sobolev norms that arise using the weighted inequality of Stein and Weiss \cite{SW58}.  Although the right-hand sides of these estimates below may appear puzzling, we must keep them until after mollification.

Recall that $\Omega\ni 0$ denotes an arbitrary bounded open subset of $\mathbb{R}^d$ and that the notation $g'$, $g''$, $g'''$, $g^{(4)}$, etc. means that the supremum of the norm in which it appears has been taken over all partial derivatives of order one, two, three, four,... respectively. 


\begin{lemma}
\label{lem:bas_rcalc}
Suppose that $F,G\in \mathcal{S}(\mathbb{R}^d: \mathbb{C})$.  Then there exists a constant $C_1 >0$ such that 
\begin{align}
\label{rc1}\| F\partial_k u\|_{(-1)} &\leq C_1(\|F\|_{\infty})\Big[ \|u\|+ \| F'u \|_{(-1)} \Big]
\end{align}
for all $u\in \mathcal{S} (\mathbb{R}^d: \mathbb{C})$.  Moreover, for every $\gamma \in (0,1)$ there exists $C_2>0$ such that 
\begin{align}
\label{rc2}\| [F\partial_k, \scb_{-1-\gamma}G\partial_j]u\| & \leq  C_2(|F|_{1}, |G|_{1}) \Big[\|u\|+ \|\,\, | F'' u|\, \, \|_{(-1)} +\|\,\, | G'' u|\, \, \|_{(-1)}\Big]
\end{align} 
for all $u\in \mathcal{S} (\mathbb{R}^d: \mathbb{C})$.   
\end{lemma}

\begin{proof}
We prove the inequalities in order.  Let $\delta >0$ be arbitrary.  Since $1+\gamma >1$, integration by parts gives 
\begin{eqnarray*}
\scb_{-1-\gamma}F\partial_k u = -(\partial_{k}(F) u * G_{1+\gamma}) + (F u * (\partial_{k}G_{1+\gamma})).  
\end{eqnarray*}
We see that by (p3) and the convolution theorem the first term above has $L^2$ norm bounded by $ \|\partial_k(F)u\|_{(-1)}$.  Similarly, the second term has $L^2$ norm bounded by $\|F\|_{\infty} \|u\|$.  Since the constants in the estimates are independent of $\gamma >0$, \eqref{rc1} now follows.

To obtain \eqref{rc2}, fix $\gamma >0$ and notice:
\begin{align*}
[F \partial_k, \scb_{-1-\gamma}G \partial_j]u&= F (\partial_k(G \partial_j u) * G_{1+\gamma}) - (G \partial_j(F \partial_k u ) * G_{1+\gamma})) \\&= F(G \partial_j \partial_k u *G_{1+\gamma}) - (F G \partial_j \partial_k u*G_{1+\gamma}) \\
& \qquad + F (\partial_k (G) \partial_j u*G_{1+\gamma}) - (G \partial_j(F) \partial_k u*G_{1+\gamma}).
\end{align*}
Using this expression, integrate by parts a few times to then see that 
\begin{align*}
[F \partial_k, \scb_{-1-\gamma}G \partial_j]u &=   F(G  \partial_k u *\partial_j G_{1+\gamma}) - (F G  \partial_k u*\partial_jG_{1+\gamma}) \\
&\qquad  -F(\partial_j(G)   u *\partial_k G_{1+\gamma}) + (\partial_j(F G)  u* \partial_k G_{1+\gamma}) \\
&\qquad  -(\partial_k\partial_j(F G)  u*  G_{1+\gamma})+ ( \partial_k(G \partial_j(F)) u* G_{1+\gamma}) \\
& \qquad + F (\partial_k (G)  u*\partial_j G_{1+\gamma}) - (G \partial_j(F) u* \partial_k G_{1+\gamma}).  
\end{align*} 
It is not hard to see that the last six terms above have $L^2$ norm bounded by 
\begin{align*}
C(|F|_{1}, |G|_{1})\big[ \|u\|+\|\,\, |F'' u|\, \,\|_{(-1)} + \|\,\, |G''u|\,\, \|_{(-1)} \big]
\end{align*}
for some constant $C>0$.  For the remaining two terms, first write
\begin{align*}
 &F(G  \partial_k u *\partial_j G_{1+\gamma}) - (F G  \partial_k u*\partial_jG_{1+\gamma}) \\
 &= \int_{\mathbb{R}^d} (F(x)-F(y))G(y) \partial_k u(y) \partial_jG_{1+\gamma}(x-y) \, dy.  
\end{align*}
Because of the difference $\Delta_F:=F(x)-F(y)$, we are permitted to use integration by parts once more to see that   
\begin{align*}
 &F(G  \partial_k u *\partial_j G_{1+\delta}) - (F G  \partial_k u*\partial_jG_{1+\gamma})=\\
&- \int_{\mathbb{R}^d}  \partial_k[(F(x)-F(y))G(y)] u(y) \partial_j G_{1+\gamma}(x-y) \, dy \\
&\qquad  + \int_{\mathbb{R}^d}  (F(x)-F(y))G(y) u(y)\partial_k\partial_j G_{1+\gamma}(x-y) \, dy.
\end{align*}
The first term above is easily seen to have the desired estimate.  For the second and last term, use the norm $|F|_{1}$ on the last term and Lemma \ref{lem:kernel_a} part (3) to obtain 
\begin{align*}
&\bigg|\int_{\mathbb{R}^d}  (F(x)-F(y))G(y) u(y)\partial_k\partial_j G_{1+\gamma}(x-y) \, dy \bigg|\\
 &\leq C \|G\|_{\infty} |F|_{1} \sum_{i=1}^j ( |u|* J_{s}) 
\end{align*}
for some $J_s\in \mathcal{J}_s$ with $s \geq \gamma>0$ and some constant $C>0$.  Standard Fourier analysis then finishes the result.    
\end{proof}


\section{Reducing subellipticity to the key estimates}
\label{sec:sub_boot}

In this section, we assume the two estimates claimed in the following lemma and use them to prove $\mathscr{L}$ is subelliptic in $\Omega$ for some $\delta>0$.  Such estimates arise naturally when attempting to bound commutators with $X$ as in the decomposition $X_0=X+ \Upsilon$ introduced in Section 2.            
\begin{lemma}\label{lem:tough_rcalc}
Fix $\alpha \in (1,2)$ and let $J\in \mathcal{J}_\alpha$.  Consider the operator $\mcm :\mathcal{S}(\mathbb{R}^d: \mathbb{C}) \rightarrow \mathcal{S}(\mathbb{R}^d: \mathbb{C})$ defined by  
\begin{align*}
\mcm u = a (b \partial_i \partial_j u *J)
\end{align*}
where $a, b \in C^{\infty}(\mathbb{R}^d: \mathbb{R})\cap B(\mathbb{R}^d: \mathbb{R})$.  Let $V=V^l \partial_l \in T\mathcal{S}(\mathbb{R}^d: \mathbb{R})$ and $\beta >0$ be such that $\alpha + \beta >2$. Then for all $Q\subset\Omega$ compact and any number $\gamma \in (2-\alpha,1)$, there exist constants $C_1, C_2>0$ such that the following estimates are valid for all $u\in C_0^{\infty}(Q: \mathbb{C})$:  
\begin{align*}
&\| [V, \mcm] u \|_{(-\beta)}\leq C_1(|V|_{1+ \gamma}) \Big[ \|u \| + \max_{l_1}\| \,\,  |(V^{l_1})'' u | \,\, \|_{(-\alpha -\beta +1)} \\& \qquad \qquad \qquad  \qquad \qquad \qquad \qquad \qquad + \max_{l_2}\| \,  \, |(V^{l_2})''' u | \,\, \|_{(-\alpha-\beta)}\Big]; \\
&\| [V,[V, \mcm]] u \|_{(-\beta)} \leq  C_2(|V|_{2+\gamma}) \Big[ \|u \| + \max_{l_1} \| \,\, | (V^{l_1})''' u| \,\, \|_{(-\alpha-\beta+1)} \\
\nonumber & \qquad \qquad \qquad \qquad  \qquad \qquad \qquad \qquad +\max_{l_2} \| \,\, | (V^{l_2})^{(4)} u| \,\, \|_{(-\alpha-\beta)}\Big].   
\end{align*}
\end{lemma} 
The proof of this lemma constitutes all of Section \ref{sec:Mcommest} which is why it is deferred until then.  Nevertheless, it illustrates the power of the psuedo-differential calculus in physical space since it can give careful dependence on the smoothness of the coefficients of the vector field $V$.          
 
To give concrete bounds on the Sobolev norms as in the lemma above, we will employ the following weighted inequality:
\begin{corollary}
\label{wcor}
Suppose that the parameters $s,t >0$ satisfy $s\leq t<\frac{d}{2}$ and $Q\subset \Omega$ is compact.  Then there exists a constant $C>0$ such that 
\begin{align*}
\bigg\| \, \frac{u}{|x|^s} \, \bigg\|_{(-t)} \leq C \| u \|
\end{align*}
for all $u\in C_0^{\infty}(Q: \mathbb{C})$.  
\end{corollary}

\begin{remark}
In particular for the right choice of parameters $s,t$, the action of the kernels can help ``erase" certain types of singularities.  
\end{remark} 

We will see in a moment that Corollary \ref{wcor} is a simple consequence of the behavior of the kernel $G_t$ near the origin and the following weighted Sobolev inequality \cite{SW58}.

\begin{theorem}[Stein-Weiss Inequality]\label{whls}
Let $0< \beta < d$, $\mu = \nu -2\beta$, $2\beta-d < \nu < d$ and suppose that $u |x|^{\nu/2} \in L^{2}(\mathbb{R}^{d},\, dx)$.  Then there exists a constant $C>0$ independent of $u$ such that 
\begin{equation}
\int_{\mathbb{R}^{d}} \bigg|\int_{\mathbb{R}^{d}}\frac{u(y)}{|x-y|^{d-\beta}}\, dy \, \bigg|^{2} |x|^{\mu}\, dx \leq C \int_{\mathbb{R}^{d}}|u(x)|^2 |x|^\nu \, dx.
\end{equation}    
\end{theorem}

\begin{proof}[Proof of Corollary \ref{wcor}]
Notice by smoothness of $G_t$ away from the origin and the formulas (p5), (p7), and (p8), there exists a constant $C>0$ independent of $u$ such that  
\begin{align*}
\| \, |x|^{-s} u \, \|_{(-t)}^2 &= \int_{\mathbb{R}^d} \bigg|\int_{\mathbb{R}^d} \frac{u(y)}{|y|^s} G_t(x-y) \, dy \bigg|^2\, dx\\
&\leq C \int_{\mathbb{R}^d}\bigg| \int_{\mathbb{R}^d}\frac{|u(y)| |y|^{-s} }{|x-y|^{d-t}} \, dy \bigg|^2\, dx.  
\end{align*}
Using the notation in the statement of Theorem \ref{whls}, pick $\beta=t$, $\mu =0$ and note that the corollary now follows since $u$ is compactly supported in $Q$.  
\end{proof}

So that we can apply the results above, throughout this section we mollify selected coefficients of $\mathscr{L}$.  Fixing $\epsilon >0$ and letting $\rho_\epsilon:\mathbb{R}^d\rightarrow [0, \infty)$ be a smooth mollifier, define
\begin{align*}
&X_0^\epsilon = (\rho_\epsilon * X_0^l) \partial_l, \qquad Y_j^\epsilon = (\rho_\epsilon * Y_j^l) \partial_l, \qquad \mathscr{L}^\epsilon = f + X_0^\epsilon + \textstyle{\sum}_{j=1}^r (Y_j^\epsilon)^* Y_j^\epsilon.  
\end{align*}
Notice that we did not change $f$ in $\scl^\epsilon$.  This is because it does not play a major role in the arguments here.  

If $u \in C_0^{\infty}(\mathbb{R}^d: \mathbb{C})$, observe that \textbf{(A1)} gives the following convergences in the $L^2$ sense as $\epsilon \downarrow 0$: 
\begin{align*}
&X_0^\epsilon u \rightarrow X_0 u, \qquad Y_j^\epsilon u \rightarrow Y_j u, \qquad 
\mathscr{L}^\epsilon u \rightarrow \mathscr{L}u.  
\end{align*}       
Moreover, note that mollification remains ``well-behaved" in the Lipschitz norm $| \cdot |_{s}$; that is, if $F \in C_0^{s}(\mathbb{R}^d: \mathbb{R})$ where $s\geq 0$, then $F^\epsilon:= (\rho_\epsilon * F)$ satisfies 
\begin{align*}
\sup_{\epsilon \in (0,1)}|F^\epsilon|_{s}&= \sup_{\epsilon \in (0,1)} \bigg(\|F^\epsilon\|_{\infty}+ \sup_{|\sigma|=\lfloor s \rfloor}\sup_{x\neq y} \frac{|D^\sigma F^\epsilon(x)-D^\sigma F^\epsilon (y)|}{|x-y|^{s-\lfloor s \rfloor}}\bigg)\\
\\
&\leq \| F\|_{\infty}+  \int_{\mathbb{R}^d} \rho_\epsilon(z) \sup_{|\sigma|=\lfloor s \rfloor}\sup_{x\neq y}\frac{|D^\sigma F(x-z)-D^\sigma F(y-z)|}{|x-z -(y-z)|^{s-\lfloor s\rfloor}} \, dz\bigg)\\
&\leq |F|_{s}.   
\end{align*}

Since we have already seen that $\mathscr{L}$ locally smooths along the directions contained in $\mathcal{F}_0$, we now show the same is true for the direction determined by the vector field $X$ where $X\in TC_0^\infty(\mathbb{R}^d: \mathbb{R})$ is as in the decomposition $X_0=X+ \Upsilon$, $X\in TC_0^\infty(\mathbb{R}^d: \mathbb{R})$, $\Upsilon \in \mathcal{F}_{0}^B$.

\begin{lemma}\label{driftlem} Suppose that $X_0= X + \Upsilon$ for some $X\in TC_0^\infty(\mathbb{R}^d: \mathbb{R})$, $\Upsilon \in \mathcal{F}_0^B$, and let $Q\subset \Omega$ be compact.  If \emph{\textbf{(A2)}} is satisfied, then for each $\gamma >0$ there exists a constant $C>0$ such that
\begin{equation}\label{driftest}
\|   X u \|_{(-\frac{1}{2}-\gamma)}+ \|   X_0 u \|_{(-\frac{1}{2}-\gamma)} \leq C( \| \scl u \| + \|u\|) 
\end{equation}
for all $u\in C_{0}^{\infty}(Q:\mathbb{C})$.  
\end{lemma}

\begin{proof}
Let $\epsilon, \gamma >0$ be arbitrary and write 
\begin{eqnarray*}
\| X_0^\epsilon u \|_{(-1/2-\gamma/2)}^2&=&\|\scb_{-1/2-\gamma/2}( X_0^\epsilon u)\|^{2}\\ &=& \big(\scb_{-1/2-\gamma/2} (X_0^\epsilon u), \scb_{-1/2-\gamma/2}(X_0^\epsilon u) \big)\\
&=& (X_0^\epsilon u,\scb_{-1-\gamma}(X_0^\epsilon u))\\
&=& \big((\mathscr{L}^\epsilon- \textstyle{\sum_{j=1}^{r}} (Y_j^\epsilon)^{*}Y_{j}^\epsilon - f)u, \scb_{-1-\gamma}(X_0^\epsilon u)\big)\\
&\leq & \| \scl^\epsilon u \| \| \scb_{-1-\gamma}(X_0^\epsilon u)\| + \|f \|_{\infty} \| u \| \|\scb_{-1-\gamma}(X_0^\epsilon u)\|\\
&&\qquad + {\textstyle\sum_{j=1}^{r}}  \| Y_{j}^\epsilon u \|  \| Y_{j}\scb_{-1-\gamma}(X_0^\epsilon u)  \|        
\end{eqnarray*}
First realize by \eqref{rc1} and the preceding remarks, we have 
\begin{align*}
\|\scb_{-1-\gamma}(X_0^\epsilon u) \| \leq \|  X_0^\epsilon u \|_{(-1)}\leq  C(|X_0^\epsilon|_{1}) \|u\| \leq C(|X_0|_1) \|u\|
\end{align*}
for some constant $C>0$ independent of $\epsilon, \gamma$.  
Notice also that $Y_j^\epsilon \scb_{-1-\gamma}X_0^\epsilon u = [Y_j^\epsilon, \scb_{-1-\gamma}X_0^\epsilon]u + \scb_{-1-\gamma} (X_0^\epsilon Y_j^\epsilon u)$.  Thus, by the above,   
\begin{align*}
\| \scb_{-1-\gamma}X_0^\epsilon Y_j^\epsilon u \| &\leq C(|X_0|_{1}) \|Y_j^\epsilon u\|
\end{align*}  
where the constant is again independent of $\epsilon, \gamma$.  
Using this and applying \eqref{rc2}, we obtain 
\begin{align*}
\|Y_j^\epsilon \scb_{-1-\gamma}X_0^\epsilon u\| & \leq  \|[Y_j^\epsilon, \scb_{-1-\gamma}X_0^\epsilon]u\| + \| \scb_{-1-\gamma}(X_0^\epsilon Y_j^\epsilon u) \|\\
&\leq  C_1(|X_0|_{1}, |Y_j|_{1})[ \| Y_j^\epsilon u\| + \|u\|]\\
&+ C_2(|X_0|_{1}, |Y_j|_{1})\textstyle{\max}_{l,m } \Big[\|\,\,|(\rho_\epsilon * X_0^l)'' u|\,\, \|_{(-1)}\\
&\qquad \qquad +\| \,\,|(\rho_\epsilon * Y_j^m)''u|\,\,\|_{(-1)}\Big]   
\end{align*}
where $C_2$ may depend on $\gamma$ but not on $\epsilon$.  
Putting all estimates together, we find that  
\begin{align*}
\| X_0^\epsilon u \|^2_{(-1/2-\gamma/2)}&\leq D_1(|X_0|_{1}, \max_j|Y_j|_{1}) \big( \|\scl^\epsilon u \|^2 + \textstyle{\sum}_{j}\|Y_j^\epsilon u \|^2 + \| u \|^2 \big) \\
& + D_2(|X_0|_{1}, \max_j|Y_j|_{1})\max_{l,m,j}\big(
\| \,\, |(\rho_\epsilon * X_0^l)''u|\, \,\|_{(-1)}\\
&\qquad \qquad \qquad \qquad \qquad \qquad +\| \,\, |(\rho_\epsilon *Y_j^l)''u|\,\,\|_{(-1)}\big).
\end{align*}
To take the limit as $\epsilon \downarrow 0$, first note that by \textbf{(A2)}
\begin{align*}
(\rho_\epsilon * X_0^l)''= (\rho_\epsilon * (X_0^l)'')\text{ and } (\rho_\epsilon * Y_j^m)''= (\rho_\epsilon * (Y_j^m)'').
\end{align*} 
In particular, we have that  
\begin{align*}
\| X_0^\epsilon u \|^2_{(-1/2-\gamma/2)}&\leq D_1(|X_0|_{1}, \max_j|Y_j|_{1})\textstyle{\sum}_{j} \big( \|\scl^\epsilon u \|^2 + \|Y_j^\epsilon u \|^2 + \| u \|^2 \big) \\
& + D_2(|X_0|_{1}, \max_j|Y_j|_{1})\max_{l,m,j}\big(
\| \,\,(\rho_\epsilon * |(X_0^l)''|)|u|\,\,\|_{(-1)}\\
&\qquad \qquad \qquad \qquad \qquad \qquad +\| \,\, (\rho_\epsilon *|(Y_j^l)''|) |u|\,\, \|_{(-1)}\big).
\end{align*}
Since the constants in the inequality above are independent of $\epsilon$, taking $\epsilon \downarrow 0$ and then applying Proposition \ref{apest} we find that  
\begin{align*}
\|X_0 u \|_{(-1/2-\gamma/2)} \leq C\Big( \|\mathscr{L}u \| + \|u \| + \max_{l,m}(\| \,\, | (X_0^l)'' u|\,\, \|_{(-1)}+\textstyle{\sum_j}\| \,\, |(Y_j^m)''u|\,\, \|_{(-1)})\Big)
\end{align*}
for some constant $C>0$ independent of $u$.  Applying Corollary \ref{wcor} using the weighted inequalities in \textbf{(A2)} finishes the estimate for $X_0$.  The bound for $X$ now easily follows as well.    
\end{proof}

Proposition \ref{apest} and Lemma \ref{driftlem} is now used as the basis of an inductive argument to estimate commutators of vector fields in $ \{ X\}\cup \mathcal{F}_0$.  Setting up the statement of the result, let $\mathcal{F}_{1}$ be the set of vector fields $U$ such that either $U=[V,W]$ with $V,W\in \mathcal{F}_{0}$ or $U=X$.  For $k\geq 2$, $\mathcal{F}_{k}$ denotes the set of vector fields $U$ such that either $U=[V,W]$, $V\in \mathcal{F}_{0}$, $W\in \mathcal{F}_{k-1}$ or $U=[V,W]$, $V\in \mathcal{F}_{1}$, $W\in \mathcal{F}_{k-2}$.

\begin{theorem}\label{bracketlem}
Let $Q\subset \Omega$ be compact and $\gamma \in (0,1)$ be arbitrary, and suppose that \emph{\textbf{(A2)}} is satisfied.  If $V_{k}\in\mathcal{F}_{k}$ and $\epsilon \leq 2^{-k}$, then there exists a constant $C>0$ such that 
\begin{equation}
\| V_{k} u\|_{(\epsilon-1-\gamma)} \leq C( \| \scl u \| + \|u \| )
\end{equation} 
for all $u \in C_0^{\infty}(Q: \mathbb{C})$.  
\end{theorem}

As discussed previously in Remark \ref{rem:8}, an immediate consequence of the proof of Theorem \ref{bracketlem} is the following:
\begin{corollary}
If \textbf{\emph{(A1)}} is satisfied, then for each $V\in \text{\emph{Lie}}(\mathcal{F}_0)$ and $Q\subset \Omega$ compact there exist constants $C,\epsilon >0$ such that 
\begin{align*}
\| V u \|_{(\epsilon-1)}\leq C( \|\scl u \| + \|u \|)
\end{align*} 
for all $u \in C_0^{\infty}(Q: \mathbb{C})$.  
\end{corollary}
Hence if we only need fields in $\text{Lie}(\mathcal{F}_0)$ to span the entire tangent space, we only need to employ the weaker base regularity assumption \textbf{(A1)}.

\begin{proof}[Proof of Theorem \ref{bracketlem}]
The proof will be done by induction on $k\geq 0$.  The case when $k=0$ follows by Proposition \ref{apest}.  Note, moreover, as a consequence of Lemma \ref{driftlem}, the case when $V_{k}=X$ is immediate.  Therefore, suppose that either $V_k\in \mathcal{F}_{k}$, $k\geq 1$, is such that $V_k=[V, W]$ where $V\in \mathcal{F}_{k-1}$, $W\in \mathcal{F}_{0}$ or $V\in \mathcal{F}_{k}$, $k\geq 2$, is such that $V_k=[V,W]$ where $V \in \mathcal{F}_{k-2}$ and $W\in \mathcal{F}_1$.  Fix $\epsilon \leq 2^{-k}$, let $\alpha = -1-\gamma +\epsilon$ and $A=\scb_{2\alpha}V_k$.  Observe that we may write for some $v,w \in C^{\infty}_{0}(\mathbb{R}^{d}: \mathbb{R})$:    
\begin{align*}
\| V_k u \|_{(\alpha)}^{2} &= (V u,  w A u) + (W u, v Au) + (Vu, WAu) - (Wu, VAu)\\
&=  (\scb_{2\epsilon-1-2\gamma} (w Vu), \scb_{-1} V_k u) + (\scb_{2\epsilon -1 -2\gamma} (vW u), \scb_{-1}(V_k u)) \\
& \, \, + \, (Vu, AWu) + (Vu, [W,A]u) - (Wu, AVu)-(Wu, [V,A]u).     
\end{align*}   
By the inequalities \eqref{pscalc1} and \eqref{pscalc2} and the inductive hypothesis, the first two terms in the last equality above have the required estimate.  Thus we have left to bound the final four terms.  Note first that
\begin{eqnarray*}
(Vu, [W,A]u)&=& (\scb_{2\epsilon -1 -2\gamma}(Vu), \scb_{-2\epsilon+1+2\gamma}([W,\scb_{2\epsilon -2 -2\gamma}V_k]u))
\end{eqnarray*}  
can be estimated as desired by induction and the inequality \eqref{pscalc4}, thus leaving $(Vu, AWu) -(Wu, AVu)-(Wu, [V,A]u)$ to bound.  For these terms, assume first that $V\in \mathcal{F}_{k-1}$, $W\in \mathcal{F}_{0}$.  Then the estimates for      
\begin{eqnarray*}
&&(Vu, AWu) - (Wu, AVu) - (Wu, [V,A]u)\\
&& \, \, \, \, \, = (\scb_{2\epsilon-1-2\gamma}(Vu), \scb_{-1}V_kWu)+ (Wu, \scb_{2\epsilon-2-2\gamma} V_k Vu)  + (Wu, [V,\scb_{2\epsilon -2 -2\gamma}V_k]u)
\end{eqnarray*}
follow by induction, the inequalities \eqref{pscalc2} and \eqref{pscalc4}, and Proposition \ref{apest}.  Now suppose that $V\in \mathcal{F}_{k-2}$, $W\in \mathcal{F}_{1}$.  By the previous argument and the Jacobi identity, we may suppose that $W=X$.  For the term $(Xu, [V, A]u)$, write
\begin{align*}
(X u, [V, A]u) = (\scb_{-1/2} X u, \scb_{1/2} [V, \scb_{2\epsilon -2 -2\gamma} V_k]u).
\end{align*}  
Note that since $k\geq 2$, this term has the claimed estimate by Lemma \ref{driftlem} and \eqref{pscalc4}, leaving $(Vu, AXu)-(Xu, AVu)$ to bound.  Recall that $X_0 = X + \Upsilon$ where $\Upsilon \in \mathcal{F}_0^B$.  Hence we may write
\begin{equation*}
X= X_0 - \Upsilon=\scl -\Upsilon-f - \sum_{j=1}^{r}Y_j^* Y_j .
\end{equation*}
Substituting the expression above for $X$ into 
\begin{eqnarray*}
(Vu, A Xu)-(X u, AVu)=(A^* Vu, X u)-(X u, AVu)
\end{eqnarray*}
we realize the terms involving $\scl$, $\Upsilon$, and $f$ can be bounded as desired as in the case when $V \in \mathcal{F}_{k-1}$, $W\in \mathcal{F}_{0}$.  Thus we have left to estimate
\begin{eqnarray*}
(A^* Vu, Y_j^{*}Y_ju) -(Y_j^*Y_j u, AVu) = (Y_j A^* Vu, Y_ju) -(Y_j u, Y_j AVu)
\end{eqnarray*}  
for $j=1,2,\ldots, r$.  
Because $\|Y_j u \|$ has the required estimate, we have left to bound $\|Y_j  M  u\|$ where $M$ is either $AV$ or $A^*V$.  Here is where we mollify and apply Lemma \ref{lem:tough_rcalc}.  Fix $\epsilon^* >0$ and consider $\|Y_j^{\epsilon^*} Mu \|$.  In light of Proposition \ref{apest} and the inductive hypothesis, there exists constants $C, C', C''$ independent of $\epsilon^*>0$ and $u$ such that
\begin{align*}  
\|Y_j^{\epsilon^*}  M  u\|^2 &\leq C( \RE(\scl^{\epsilon^*} Mu, Mu) + \|Mu \|^2)\\
&\leq  C'( \RE((\scl^{\epsilon^*}-f) Mu, Mu) +  \|Mu \|^2)\\
& = C'\Big( \RE( \scb_{-2\epsilon} M (\scl^{\epsilon^*}-f) u, \scb_{2\epsilon} M u) \\
&\qquad + \RE( \scb_{-2\epsilon}[\scl^{\epsilon^*}-f, M]u , \scb_{2\epsilon} M u)  + \|Mu\|^2\Big)\\
&\leq C''( \| (\scl^{\epsilon^*}-f) u \|^2 +\|\scl u\|^2+ \|u \|^2 + \|\scb_{-2\epsilon} [\scl^{\epsilon^*}-f, M]u \|^2).  
\end{align*}
Since all other terms will have the correct estimate when taking $\epsilon \downarrow 0$, we now focus our attention on the last term $\|\scb_{-2\epsilon} [\scl^{\epsilon^*}-f, M]u \|^2$.  Here we seek to apply Lemma \ref{lem:tough_rcalc}.  First write
\begin{align*}
\scl^{\epsilon^*}-f = - \sum_{j=1}^r (Y_j^{\epsilon^*})^2 + \hat{X}^{\epsilon^*}_{0}
\end{align*}
where 
\begin{align*}
\hat{X}^{\epsilon^*}_{0}= X_0^{\epsilon^*} - \sum_{j=1}^r (\rho_{\epsilon^*} * \partial_l Y_j^l) Y_j^{\epsilon^*}.  
\end{align*}
Notice $M$ is the finite sum of terms of the required form of Lemma \ref{lem:tough_rcalc}.  Thus, letting 
\begin{align*}
g_{\epsilon*}^l= (X_0^l* \rho_{\epsilon^*})-\sum_{j=1}^r (\partial_m Y^m_j *\rho_{\epsilon^*}) (Y^{l}_j * \rho_{\epsilon^*}), 
\end{align*}
now apply Lemma \ref{lem:tough_rcalc} to see that 
\begin{align*}
&\|\scb_{-2\epsilon} [\scl^{\epsilon^*}-f, M]u \|\\
&\leq \sum_{j=1}^r \|\scb_{-2\epsilon} [(Y_j^{\epsilon^*})^2, M]u \| +  \|\scb_{-2\epsilon} [\hat{X}_0^{\epsilon^*}, M]u \| \\
&\leq \sum_{j=1}^r \|\scb_{-2\epsilon} [Y_j^{\epsilon^*}, M]Y_j^{\epsilon^*}u \| + 2 \|\scb_{-2\epsilon} [Y_j^{\epsilon^*},[Y_j^{\epsilon^*}, M]]u \| +  \|\scb_{-2\epsilon} [\hat{X}_0^{\epsilon^*}, M]u \|\\
&\leq \sum_{j=1}^r C(|X_0|_{3/2}, |Y_j|_{5/2}) \times \\
&\qquad \bigg[ \| u\| + \max_{l} \| \, \,| (g_{\epsilon^*}^{l})'' u| \,\, \|_{(-1)} + \max_{l} \|\,\,|(g_{\epsilon^*}^{l})''' u| \, \, \|_{(-2)}+\|Y_j^{\epsilon^*}u\|\\
&\qquad + \max_{l}\| \, \, |(Y_j^{l} * \rho_{\epsilon^*})''' u| \, \, \|_{(-1)} + \max_{l}\| \, \, |(Y_j^{l} * \rho_{\epsilon^*})^{(4)} u| \, \, \|_{(-2)}\\
& \qquad + \max_{l}\| \, \, |(Y_j^{l} * \rho_{\epsilon^*})'' Y_j^{\epsilon*} u| \, \, \|_{(-1)}  + \max_{l}\| \, \, |(Y_j^{l} * \rho_{\epsilon^*})''' Y_j^{\epsilon^*}u| \, \, \|_{(-2)}  \bigg].
\end{align*}
By \textbf{(A2)}, we may exchange derivatives with integration in the convolution terms.  Thus, taking the limit as $\epsilon \downarrow 0$ and then applying the weighted inequalities in \textbf{(A2)} finishes the proof.    
\end{proof}

We now use the previous result to prove that $\scl$ is subelliptic under the hypotheses of Theorem \ref{thm:main}.

\begin{theorem}
\label{thm:subelliptic}
Suppose either $\textbf{\emph{(A1)}}$ is satisfied and $\text{\emph{Lie}}_{x}(\mathcal{F}_0)= \mathbb{R}^d$ for all $x\in \Omega$ or $\textbf{\emph{(A2)}}$ is valid and $\text{\emph{Lie}}_{x}(  \{X\}\cup \mathcal{F}_0)= \mathbb{R}^d$ for all $x\in \Omega$.  Then for all $Q\subset \Omega$ compact there exist constants $\delta \in (0,1]$, $C>0$ such that 
\begin{align}
\|u\|_{(\delta)}\leq C ( \|\scl u \| + \|u\|)
\end{align}
for all $u \in C_{0}(Q: \mathbb{C})$.  
\end{theorem}  
\begin{proof}
We follow the arguments given on p. 196-197 of \cite{Str12}.  Fix $Q\subset \Omega$ compact.  By hypothesis, for each $x\in \Omega$ there exist vector fields $X_1^{x}, \ldots, X_d^{x} \in TC_0^{\infty}(\mathbb{R}^d: \mathbb{R}) $ such that 
\begin{align*}
\text{span}_{\mathbb{R}}\{X_1^{x}(x), \ldots, X_d^{x}(x) \}= \mathbb{R}^d.  
\end{align*}
Considering the matrix $A(x,y)$ with columns $X_1^x(y), \ldots, X_d^x(y)$ we see that for each $x\in\Omega$, $\text{det}(A(x,x))\neq 0$.  In particular since each $X_i^x $ has smooth coefficients, the set 
$U(x)= \{y \in \Omega\, : \,  \text{det}(A(x,y)) \neq 0 \}$ is non-empty and open.  Moreover,
\begin{align*}
Q\subset \bigcup_{x\in Q} U(x).
\end{align*}  
Since $Q$ is compact, we may extract a finite sub-cover $U(x_1), \ldots , U(x_n)$ of $Q$.  Hence by Theorem \ref{bracketlem}, there exist smooth vector fields $Z_1, Z_2, \ldots, Z_k$ and constants $\delta \in (0,1],\,  C>0$ such that
\begin{align*}
\| Z_l u \|_{(\delta-1)} \leq C( \|\scl u \| + \|u\|)
\end{align*}
for all $l=1,2, \ldots, k$, $u\in C_0^\infty(Q: \mathbb{C})$ and such that for $j=1,2, \ldots, d$
\begin{align*}
\partial_j = \textstyle{\sum}_{l=1}^k a_l(x) Z_l, \qquad x\in Q,   
\end{align*} 
where each $a_l \in C^{\infty}(Q: \mathbb{R})$.  The result now follows.   
\end{proof}

\section{Proof of Theorem \ref{thm:main} assuming subellipticity}


In this section, we prove Theorem \ref{thm:main} assuming $\mathscr{L}$ is subelliptic in $\Omega$ of some parameter $\delta >0$.  Because we need to be able to work with distributions $v\in \Bp$ as opposed to smooth functions, we begin by deriving some auxiliary results involving mollifiers and regularized kernels.  Thus let $\rho \in C_0^{\infty}(\mathbb{R}^d: [0, 1])$ satisfy $\rho\equiv 1$ on $|x|\leq 1$ and $\rho \equiv 0$ on $|x|\geq 2$.  For $\epsilon, s >0$, define $\rho_\epsilon(x)=\epsilon^{-d} \rho(\epsilon^{-1}x)$ and set 
\begin{align*}
G_s^{\epsilon}(x) = (\rho_{\epsilon} * G_s)(x).      
\end{align*}
It is easy to see that $G_s^{\epsilon} \in C^{\infty}(\mathbb{R}^d:[0, \infty))\cap L^1(\mathbb{R}^d, dx)$.  The kernel $G_s^\epsilon$ will serves as a smooth approximation to $G_s$. 

We first need the following proposition:  

\begin{proposition}
\label{prop:product}
Let $s,t >0$ satisfy $s+t >d$ and fix $\epsilon >0$.  Then the product functions $G_s G_t$ and $G_s^\epsilon G_t$ belong to $L^1(\mathbb{R}^d: [0, \infty))$.  Moreover
\begin{align*}
\widehat{G_s G_t}(\xi) = \int_{\mathbb{R}^d} \widehat{G_s}(\xi-\eta) \widehat{G_t}(\eta) \, d\eta, \qquad \widehat{G_s^\epsilon G_t}(\xi) = \int_{\mathbb{R}^d} \widehat{G_s^\epsilon}(\xi-\eta) \widehat{G_t}(\eta) \, d\eta
\end{align*}    
\end{proposition}

\begin{remark}
Because $\widehat{G_s}$ and $\widehat{G_t}$ need not belong to $L^1(\mathbb{R}^d: \mathbb{C})$, the usual convolution theorem does not immediately apply here.  
\end{remark}

\begin{proof}
Since $s+t >d$, $G_s G_t, \, G_s^\epsilon G_t \in L^1(\mathbb{R}^d, [0, \infty))$ by smoothness of the kernels away from the origin and the asymptotic formulas (p7) and (p8).  To compute their Fourier transforms, letting $\varphi \in \mathcal{S}(\mathbb{R}^d: \mathbb{C})$ and $\epsilon_1, \epsilon_2 >0$ we make use of the relation
\begin{align}
\label{eqn:Gconv}
\int_{\mathbb{R}^d} \widehat{G_s^{\epsilon_1}G_t^{\epsilon_2}}(\xi) \varphi(\xi) \, d\xi = \int_{\mathbb{R}^d} G_s^{\epsilon_1}(x) G_t^{\epsilon_2}(x) \hat{\varphi}(x)\, dx.  
\end{align}
First observe that
\begin{align*}
\int_{\mathbb{R}^d} G_s^{\epsilon_1}(x) |G_t^{\epsilon_2}(x)-G_t(x)| |\hat{\varphi}(x)|\, dx & \leq C \| G_t^{\epsilon_2}- G_t\|_{1}\rightarrow 0
\end{align*}
as $\epsilon_2 \downarrow 0$ since $G_t \in L^1(\mathbb{R}^d: [0, \infty))$.   In particular, the right-hand side of \eqref{eqn:Gconv} approaches $\int_{\mathbb{R}^d} G_s^{\epsilon_1}(x) G_t(x) \hat{\varphi}(x)\,dx$ as $\epsilon_2\downarrow 0$.  To see what happens to the left-hand side, note first that since we have regularized $G_s$ and $G_t$ we may write 
\begin{align*}
 \int_{\mathbb{R}^d} \widehat{G_s^{\epsilon_1}G_t^{\epsilon_2}}(\xi) \varphi(\xi) \, d\xi = \int_{\mathbb{R}^d} \int_{\mathbb{R}^d} \hat{\rho}(\epsilon_1(\xi-\eta)) \langle \xi-\eta\rangle^{-s} \hat{\rho}(\epsilon_2 \eta) \langle \eta\rangle^{-t}\,  \varphi(\xi) \, d\eta \, d\xi  
\end{align*}
as $\hat{\rho} \in \mathcal{S}(\mathbb{R}^d: \mathbb{C})$.  
Since $s+t >d$ and $\varphi \in \mathcal{S}(\mathbb{R}^d: \mathbb{C})$, by the dominated convergence theorem we may take the limit as $\epsilon_2 \downarrow 0$ inside both integrals on the right-hand side above to obtain the desired formula for $\widehat{G_s^{\epsilon_1}G_t}$.  The derivation of the formula for $\widehat{G_t G_s}$ follows exactly in the same manner.      
\end{proof}

By the structure of distributions belonging to $\Bp$, we recall that (see for example \cite{Hor03}) any $v\in \Bp$ can be identified with a positive measure $m_v$ on $\Omega$ through its distributional pairing $\langle\cdot, v \rangle$; that is, we have the following for all $\phi \in B_0(\Omega: \mathbb{R})$:
\begin{align*}
\langle\phi, v \rangle = \int_{\Omega} \phi(x) \, m_{v}(dx) < \infty.
\end{align*} 
This is an extremely important observation as it allows for a number of conveniences in this section.

    Using the previous proposition, we now show how we plan to bound remainder terms that will arise in various situations.  
\begin{lemma}
\label{lem:supdist}

Let $\alpha \geq 0$, $s>0$, and $\sigma$ be a multi-index satisfying $s\notin \mathbb{N}$ and the relationship $\alpha + s- |\sigma| >0$.  If $v\in \Bp$ and $\varphi \in C_0^\infty(\Omega: \mathbb{R})$, then there exists a constant $C>0$ independent of $v$ such that 
\begin{align}
\sup_{\epsilon \in (0,1)}\bigg\| \int_{\mathbb{R}^d} |x-y|^\alpha |D^\sigma G_s^\epsilon(x-y)|\,| \varphi(y)| dm_{v}(y) \bigg\| \leq C \| \psi v \|_{(|\sigma|-\alpha - s)}  
\end{align}   
for some $\psi \in C^\infty_0(\Omega: \mathbb{R})$.  

\end{lemma}
\begin{proof}
In this proof, $C$ will be used to denote a generic positive constant independent of $\epsilon$; its value may change from line to line.  There are two cases: $\lfloor s \rfloor > |\sigma|=:k$ and $\lfloor s \rfloor \leq |\sigma|=k$.  Suppose first that $\lfloor s \rfloor > |\sigma|=k$.  Then it follows by Lemma \ref{lem:kernel_a} that     
\begin{align*}
|x|^\alpha |D^\sigma G_s^{\epsilon}(x)| &= |x|^\alpha \bigg|\int_{\mathbb{R}^d} \rho_\epsilon(x-y) (D^\sigma G_s)(y) \, dy \bigg|\\
&\leq C |x|^\alpha \int_{\mathbb{R}^d} \rho_\epsilon(x-y ) J_{s-k}(y) \, dy
\end{align*}
for some non-negative $J_{s-k} \in \mathcal{J}_{s-k}$.  But note that 
\begin{align*}
|x|^\alpha \int_{\mathbb{R}^d} \rho_\epsilon(x-y ) J_{s-k}(y) \, dy &\leq C  \int_{\mathbb{R}^d} |x-y|^\alpha \rho_\epsilon(x-y ) J_{s-k}(y) \,dy\\
&\qquad + C \int_{\mathbb{R}^d} \rho_\epsilon(x-y ) |y|^\alpha J_{s-k}(y) \, dy \\
&\leq  C \epsilon^\alpha (\rho_\epsilon * J_{s-k})(x) + C(\rho_\epsilon * K_{s+\alpha-k})(x) 
\end{align*}
for some $K_{s+\alpha-k} \in \mathcal{J}_{s+\alpha-k}$.  Hence in this case if $\psi \in C_0^\infty(\Omega: [0, \infty)$ with $\psi \equiv 1$ on $\text{supp}(\varphi)$ we have
\begin{align*}
 &\bigg\| \int_{\mathbb{R}^d} |x-y|^\alpha |D^\sigma G_s^{\epsilon}(x-y)||\varphi(y)|\, dm_{v}(y) \bigg\|^2\\
 &\leq \bigg\| \int_{\mathbb{R}^d} |x-y|^\alpha |D^\sigma G_s^{\epsilon}(x-y)|\psi(y)\, dm_{v}(y) \bigg\|^2\\
 & \leq C \int_{\mathbb{R}^d} \epsilon^{2\alpha} |\hat{\rho}(\epsilon \xi)|^2 (1+ 4\pi^2|\xi|^2)^{k-s} |\widehat{\psi v}(\xi)|^2 \, d\xi \\
 & \qquad +  C \int_{\mathbb{R}^d} |\hat{\rho}(\epsilon \xi)|^2(1+ 4\pi^2|\xi|^2)^{k-s-\alpha} |\widehat{\psi v}(\xi)|^2 \,ci d\xi  \\
 &\leq C  \int_{\mathbb{R}^d}(1+4\pi^2 \epsilon^2 |\xi|^2)^{\alpha}|\hat{\rho}(\epsilon \xi)|^2(1+ 4\pi^2|\xi|^2)^{k-s-\alpha} |\widehat{\psi v}(\xi)|^2 \, d\xi \\
 & \qquad +  C \int_{\mathbb{R}^d}|\hat{\rho}(\epsilon \xi)|^2(1+ 4\pi^2|\xi|^2)^{k-s-\alpha} |\widehat{ \psi v}(\xi)|^2 \, d\xi \\
 &\leq C \|\psi v\|_{(k-s-\alpha)}^2,  
\end{align*}
finishing the proof in this case.  

Now suppose that $\lfloor s \rfloor \leq k =|\sigma|$.  Let $\tau\leq \sigma$ be a multi-index with $|\tau|= \lfloor s \rfloor$ and notice now that since $s> \lfloor s \rfloor$
\begin{align*}
|x|^\alpha |D^\sigma G_s^\epsilon(x)| &= |x|^\alpha \bigg|\int_{\mathbb{R}^d}D_x^{\sigma-\tau}( \rho_\epsilon(x-y)) (D^\tau G_s)(y) \, dy \bigg|\\
&= |x|^{\alpha-(k- \lfloor s \rfloor)} |x|^{k- \lfloor s \rfloor}\bigg|\int_{\mathbb{R}^d}D_x^{\sigma-\tau}( \rho_\epsilon(x-y)) (D^\tau G_s)(y) \, dy \bigg|\\
&\leq   C |x|^{\alpha-(k- \lfloor s \rfloor)} \sum_{i=1}^d   |x_i|^{k- \lfloor s \rfloor}\bigg|\int_{\mathbb{R}^d}D_x^{\sigma-\tau}( \rho_\epsilon(x-y)) (D^\tau G_s)(y) \, dy \bigg|.
\end{align*}  
Notice  
\begin{align*}
 &|x_i|^{k- \lfloor s \rfloor}\bigg|\int_{\mathbb{R}^d}D_x^{\sigma-\tau}( \rho_\epsilon(x-y)) (D^\tau G_s)(y) \, dy \bigg|\\
 &\leq \sum_{j_1+j_2=k-\lfloor s\rfloor }C\bigg| \int_{\mathbb{R}^d} (x_i-y_i)^{j_1} D^{\sigma-\tau} (\rho_\epsilon(x-y)) y_i^{j_2} (D^\tau G_s)(y) \, dy \bigg|.     
\end{align*}
By inducting on $j_1=0,\ldots, k-\lfloor s \rfloor $, one can integrate by parts to deduce the following estimate
\begin{align*}
|x|^\alpha |D^\sigma G_s^\epsilon(x)| \leq C |x|^{\alpha-(k- \lfloor s \rfloor)} \sum_{l=1}^M G_{s_l}^{\epsilon}(x)
\end{align*}
where $s_l \geq s-\lfloor s \rfloor$.   If $\alpha -(k-\lfloor s \rfloor)\geq 0$, we can use the same line of reasoning in the first case to establish the result.  If, however, $\alpha -(k-\lfloor s \rfloor)<0$ first notice that $-r:=\alpha -(k-\lfloor s \rfloor) >-1$.  But since $m_v$ is a positive measure, we see that for $\psi \in C_0^\infty(\Omega: [0, \infty))$ with $\psi \equiv 1 $ on $\text{supp}(\varphi)$:
\begin{align*}
\bigg\| \int_{\mathbb{R}^d} \frac{G_{s_l}^{\epsilon}(x-y)}{|x-y|^r}\, |\varphi(y)|dm_{v}(y)\bigg\|& \leq   C \bigg\| \int_{|x-y|< 1} \frac{G_{s_l}^{\epsilon}(x-y)}{|x-y|^r}\, \psi(y) dm_{v}(y)\bigg\|\\
&+ \bigg\| \int_{|x-y|\geq 1} \frac{G_{s_l}^{\epsilon}(x-y)}{|x-y|^r}\, \psi(y) dm_{v}(y)\bigg\| \\
&\leq C \bigg\| \int_{\mathbb{R}^d}G_{s_l}^\epsilon(x-y)G_{d-r}(x-y)\, \psi(y) dm_{v}(y)\bigg\|\\
&+ C\bigg\| \int_{\mathbb{R}^d}G_{s_l}^{\epsilon}(x-y)\, \psi(y) dm_{v}(y)\bigg\|. 
\end{align*}  
First observe that 
\begin{align*}
\bigg\| \int_{\mathbb{R}^d}G_{s_l}^{\epsilon}(x-y)\, \varphi(y) dm_{v}(y)\bigg\| \leq \| \varphi v\|_{(-s_l)}\leq \| \varphi v\|_{(k-\alpha-s)}.
\end{align*}
For the remaining term, apply Proposition \ref{prop:product} and non-negativity of $v$ to see that for $\psi \in C_0^\infty(\Omega: [0, \infty))$ with $\psi \equiv 1$ on $\text{supp}(\varphi)$:
\begin{align*}
& \bigg\| \int_{\mathbb{R}^d}G_{s_l}^\epsilon(x-y)G_{d-r}(x-y)\,\psi(y) dm_{v}(y)\bigg\|\\
& \leq  C \bigg\| \int_{\mathbb{R}^d}G_{s_l}(x-y)G_{d-r}(x-y)\, \psi(y) dm_{v}(y)\bigg\|\\
&\leq  C \bigg\| \int_{\mathbb{R}^d}G_{\alpha + s - k}(x-y)\, \psi(y) dm_{v}(y)\bigg\|\\
& = C \| \psi v \|_{(k-\alpha - s)}. 
\end{align*}
Note that the penultimate line above follows by smoothness of the kernels away from the origin and the asymptotic formulas (p7) and (p8).  This finishes the proof.  
\end{proof}

Recalling that $\scl= \sum_{j=1}^r Y_j^* Y_j + X + f$, we now use the previous lemma to establish the following:
\begin{lemma}
\label{lem:52}
Under the hypotheses of Theorem \ref{thm:main}, fix $s< \bS$ and suppose that $u, \scl u \in H^{s}_{\emph{\text{loc}}}(\Omega:\mathbb{R})$.  For $\epsilon >0$, define $A^\epsilon_s= \psi \widetilde{\scb}_{s}^\epsilon \eta$ where $\psi, \eta \in C_0^\infty(\Omega: \mathbb{R})$ and $\widetilde{\scb}_{s}^\epsilon$ is the following operator  
\begin{align*}
\widetilde{\scb}_{t}^\epsilon u= (u * x^\tau D^\sigma G_{t}^\epsilon) 
\end{align*}
where $t>0$ and $\sigma$ and $\tau$ are multi-indices satisfying $s + |\tau|-|\sigma| \geq - s$.  Then 
\begin{align*}
\sup_{\epsilon \in (0,1)} \|[f, A^\epsilon_s] v \|, \, \, \, \sup_{\epsilon \in (0,1)} \|[X_0, A^\epsilon_s] v \|, \, \, \, \sup_{\epsilon \in (0,1)} \|Y_j A^\epsilon_s v \|
\end{align*}   
are all finite.  
\end{lemma}

\begin{proof}
Again, $C$ will denote a positive constant independent of $\epsilon$.  Here we write for simplicity $J_{-s}^\epsilon=x^\tau D^\sigma G_t^\epsilon$ and start by showing $\sup_{\epsilon \in (0,1)}\|  [f,A^{\epsilon}_s] v \|< \infty $.  Let $k$ be the smallest non-negative integer such that $k >s$.  Then  
  \begin{align*}
 - [f, A^{\epsilon}_s]v &= \psi(x) \int_{\mathbb{R}^d} (f(y)-f(x)) J_{-s}^\epsilon (x-y) \eta(y) \, dm_v(y)\\
  &=\psi(x) \int_{\mathbb{R}^d} \bigg( f(y)-\sum_{|\upsilon|\leq \max(k-1,0)}\frac{D^\upsilon f(x)}{\upsilon !}(y-x)^\upsilon\bigg)   J_{-s}^\epsilon (x-y) \eta(y) \, dm_v(y)\\
 & + \psi(x)\sum_{1 \leq |\upsilon|\leq \max(k-1,0)}\frac{D^\upsilon f(x)}{\upsilon !} \int_{\mathbb{R}^d}   (y-x)^\upsilon  J_{-s}^\epsilon(x-y) \eta(y) \, dm_v(y).
  \end{align*}
  Since $f\in C^{s^+}_0(\mathbb{R}^d: \mathbb{R})$ for all $s<s^+<s^*$, we now find by Lemma \ref{lem:supdist} and the correspondence in Fourier space that 
   \begin{align*}
  \sup_{\epsilon \in (0,1)}\|[f, A_s^{\epsilon}]v\| < \infty.  
  \end{align*}
  Next we turn our attention to bounding $\sup_{\epsilon \in (0,1)}\|  [X_0 ,A_s^{\epsilon}] v \|$.  Interpreting derivatives in the weak sense, realize first that
  \begin{align*}
  |[X_0 ,A_s^{\epsilon}] v (x)|&\leq C \bigg| \int_{\mathbb{R}^d} (X^l_0(x)-X^l_0(y)) \partial_l J _{-s}^{\epsilon}(x-y) \eta(y) \, dm_v(y)\bigg|\\
&+ C \bigg| \int_{\mathbb{R}^d} \partial_l(X^l_0(y)\eta(y)) J_{-s}^{\epsilon}(x-y) \tilde{\eta}(y) \, dm_v(y)\bigg|  \\
  &+ C T^\epsilon v(x)
  \end{align*}
  where $\tilde{\eta}\in C_0^{\infty}(\Omega: [0,1])$ is such that $\tilde{\eta}=1$ on $\supp(\eta)$ and $\sup_{\epsilon \in (0,1)}\|T^\epsilon v\|< \infty.$  Since $X_0 \in T C_0^{s^+ +1}(\mathbb{R}^d: \mathbb{R})$ for any $s<s^+<s^*$, we can do the same Taylor formula trick as in the case of $[f, A_s^\epsilon]v$ above on both remaining terms  and conclude that 
$$\sup_{\epsilon \in (0,1)}\|  [X_0 ,A_s^{\epsilon}] v \|< \infty.$$

  The estimate for $\|Y_j A_s^\epsilon u\|$ is more involved.  First notice by Proposition \ref{apest} and the previous two estimates we have by Cauchy-Schwarz and the hypothesis $v, \scl v \in H^{s}_{\text{loc}}(\Omega: \mathbb{R})$  
  \begin{align*}
\sum_{j=1}^r  \|Y_j A_s^\epsilon v\|^2 & \leq \RE(\scl A_s^\epsilon v, A_s^\epsilon v ) + C \| A_s^\epsilon v\|^2 \\
  &= \RE(A_s^\epsilon \scl v, A_s^\epsilon v) + \RE([\scl, A_s^\epsilon]v, A_s^\epsilon v) + C\|A_s^\epsilon v \|^2 \\
  &\leq C+ C \sum_{j=1}^r ([Y_j^*Y_j, A_s^\epsilon]v, A_s^\epsilon v). 
  \end{align*}
  Notice that $Y_j^*Y_j=-Y_j^2 + \Tilde{Y}_j$ for some vector field $\tilde{Y}_j \in TC_0^{s^+ +1}(\mathbb{R}^d: \mathbb{R})$ where $s<s^+ <\bS$.  By the estimate for $[X_0, A_s^\epsilon]v$, we then see that 
 \begin{align*}
  \sum_{j=1}^r\|Y_j A_s^\epsilon v\|^2 &\leq C + C \sum_{j=1}^r|([Y_j^2, A_s^\epsilon]v, A_{s}^\epsilon v)|\\
  &\leq C+C\sum_{j=1}^r|(2Y_j[Y_j, A_s^\epsilon]v, A_{s}^\epsilon v)| +|([Y_j,[Y_j, A_s^\epsilon]]v, A_s^\epsilon v)|\\
  &\leq C(c)+ \frac{D}{c} \sum_{j=1}^r \| Y_j A_s^\epsilon v \|^2  + E\|[Y_j, [Y_j, A_s^\epsilon]]v\|^2  
  \end{align*}  
  for every $c>0$ for some constants $C(c), D, E>0$ where $D,E>0$ do not depend on $c$ but $C(c)\rightarrow \infty$ as $c\rightarrow \infty$.  Upon choosing $c>0$ sufficiently large, it suffices to show that 
  \begin{align*}
 \sup_{\epsilon \in (0,1)} \|[Y_j, [Y_j, A_s^\epsilon]]u \| < \infty.  
  \end{align*}    
By linearity, to prove the above it suffices to show
 \begin{align*}
 \sup_{\epsilon \in (0,1)} \|[f\partial_l, [g \partial_m, A_s^\epsilon]]v \| < \infty.  
  \end{align*} 
  where $f,g\in C_0^{s^+ +2}(\mathbb{R}^d: \mathbb{R})$, $s<s^+< \bS$.  To show this, it is helpful to apply the product rule all the way through the convolution operator $\widetilde{\scb}_{s}^{\epsilon}$.  For example, we expand each term as follows:
  \begin{align*}
  g\partial_m( \psi \widetilde{\scb}_{s}^{\epsilon} \eta v) = g \partial_m(\psi) \widetilde{\scb}_{s}^{\epsilon} \eta v + g \psi\widetilde{\scb}_{s}^{\epsilon} \partial_m(\eta) v + g \psi \widetilde{\scb}_{s}^{\epsilon}\eta \partial_m(v)
  \end{align*}  
  where of course all derivatives are interpreted in the weak sense with respect to the distributional pairing $\langle \cdot , \cdot \rangle. $  Doing this to the full operator we obtain 
  \begin{align*}
  &[f\partial_l, [g \partial_m, A_s^\epsilon]]v(x)\\
  &= \psi \widetilde{\scb}_{s}^{\epsilon}(f(x)-f(\cdot))(g(x)-g(\cdot)) \eta \partial_l \partial_m(v) \\
  &+g\partial_m(\psi) \widetilde{\scb}_{s}^{\epsilon} (f(x)-f(\cdot))\eta \partial_l(v)+  g \psi \widetilde{\scb}_{s}^{\epsilon} (f(x)-f(\cdot))\partial_m(\eta)\partial_l(v) \\
  & +   f\psi\widetilde{\scb}_{s}^{\epsilon} (\partial_l(g)(x)-\partial_l(g)(\cdot))\eta \partial_m(v) +   f\partial_l(\psi)\widetilde{\scb}_{s}^{\epsilon} (g(x)-g(\cdot))\eta \partial_m(v)\\
  & +   f\psi \widetilde{\scb}_{s}^{\epsilon}(g(x)-g(\cdot))\partial_l(\eta) \partial_m(v) +  \psi \widetilde{\scb}_{s}^{\epsilon}(g(x)-g(\cdot))\partial_l(f) \eta \partial_m(v)  
 \\
 &+ T_0^{\epsilon}v(x) 
  \end{align*}
  where $\sup_{\epsilon \in (0,1)}\|T_0^\epsilon v\| < \infty$.  After integrating by parts and applying Taylor's formula on each term, we obtain the claimed estimate since $f,g \in C_0^{s^+ + 2}(\mathbb{R}^d:\mathbb{R})$ for any $s<s^+<\bS$.  This finishes the proof of the lemma.   
\end{proof}

Finally we conclude this section by proving the following result which establishes Theorem \ref{thm:main} assuming the estimates of Lemma \ref{lem:tough_rcalc} are true.  
\begin{theorem}
Under the hypotheses of Theorem \ref{thm:main}, fix $s< \bS$ and suppose that $v\in \Bp$.  If $\scl $ is subelliptic of order $\delta >0$ in $\Omega$ then $$v, \scl v \in H^{s}_{\emph{\text{loc}}}(\Omega: \mathbb{R})\implies u \in H^{s+\delta}_{\emph{\text{loc}}}(\Omega:\mathbb{R}).$$  
\end{theorem}

\begin{proof}
Letting $\psi \in C_0^\infty(\Omega: \mathbb{R})$ be arbitrary, our goal is to show that $$\|\psi v \|_{(s+\delta)}<\infty $$ 
for $s< \bS$, $s\notin \mathbb{Z}$.  Note that we need not show the result for integer values of $s$.  Let $k$ be the smallest non-negative integer strictly larger $s$ and notice by Fatou's lemma applied in Fourier space
\begin{align*}
\|\psi v \|_{(s+\delta)}& \leq C\sum_{|\sigma|\leq k} \limsup_{\epsilon \rightarrow 0} \| D^\sigma \scb_{s-k}^{\epsilon} (\psi v ) \|_{(\delta)}.
\end{align*}
Fixing $|\sigma|\leq k$, our goal now is to show that $$\limsup_{\epsilon \rightarrow 0}\| D^\sigma \scb_{s-k}^{\epsilon} (\psi v ) \|_{(\delta)}< \infty.$$ We will accomplish this by bounding $\| D^\sigma \scb_{s-k}^{\epsilon} (\psi u ) \|_{(\delta)}$ independent of $\epsilon >0$.  Notice that   
\begin{align*}
 \| D^\sigma \scb_{s-k}^{\epsilon} (\psi v ) \|_{(\delta)}\leq  \| \psi D^\sigma \scb_{s-k}^{\epsilon} (\eta v ) \|_{(\delta)}+\| [D^\sigma \scb_{s-k}^{\epsilon}, \psi] (\eta v ) \|_{(\delta)}
\end{align*}
where $\eta$ is any function with $\eta \in C_0^{\infty}(\Omega: [0,1])$ and $\eta\equiv 1$ on $\supp(\psi)$.  Letting $\mathcal{A}_s^{\epsilon} = \psi D^\sigma \scb_{s-k}^{\epsilon} \eta$ and using the fact that $\scl$ is subelliptic of order $\delta \in (0,1)$ in $\Omega$, we obtain   
\begin{align*}
 \| D^\sigma \scb_{s-k}^{\epsilon} (\psi v)  \|_{(\delta)}&\leq  \| \mathcal{A}_{s}^{\epsilon} v  \|_{(\delta)}+\| [D^\sigma \scb_{s-k}^{\epsilon}, \psi] (\eta v ) \|_{(\delta)}\\
&\leq  C(\| \scl \mathcal{A}_s^{\epsilon} v  \| + \|\mathcal{A}_s^{\epsilon} v \|) + \| [D^\sigma \scb_{s-k}^{\epsilon}, \psi] (\eta v ) \|_{(\delta)}\\
&\leq C (\| \mathcal{A}_s^{\epsilon} \scl v  \| + \|  [\scl,\mathcal{A}_s^{\epsilon}] v  \|+ \|\mathcal{A}_s^{\epsilon} v \|) \\
&\qquad + \| [D^\sigma \scb_{s-k}^{\epsilon}, \psi] (\eta v ) \|_{(\delta)}\\
&\leq C(1 + \|  [\scl,\mathcal{A}_s^{\epsilon}] v  \|+\| [D^\sigma \scb_{s-k}^{\epsilon}, \psi] (\eta v ) \|_{(\delta)})  
\end{align*}
where all constants above are independent of $\epsilon >0$ and on the last line we have used the assumption that $u, \scl u \in H^{s}_{\text{loc}}(\Omega)$.  
Thus we have left to show that    
\begin{align*}  
  \sup_{\epsilon \in (0,1)}\|  [\scl,\mathcal{A}_s^{\epsilon}] v  \|, \qquad   \sup_{\epsilon \in (0,1)}\| [D^\sigma \scb_{s-k}^{\epsilon}, \psi] (\eta v ) \|_{(\delta)}.
\end{align*}
First observe that 
\begin{align*}
\| [D^\sigma \scb_{s-k}^{\epsilon}, \psi] (\eta v) \|_{(\delta)} \leq \sum_{|\tau|\leq 1 }\| D^\tau[D^\sigma \scb_{s-k}^\epsilon, \psi](\eta v )\|. 
\end{align*}
Using the ideas in the proof of the previous proposition, this term is easily seen to be bounded independent of $\epsilon$.  Thus applying Lemma \ref{lem:52}, all we have left to do is bound $\|[Y^{l,m}_j \partial_l \partial_m, \mathcal{A}_s^\epsilon]v\|$ where $Y^{l,m}_j=Y^l_j Y^m_j$ and $Y^m_j$ is as in $Y_j=Y^m_j \partial_m$.  Interpreting derivatives in the weak sense, decompose this term as follows  
\begin{align*}
[Y^{l,m}_j \partial_l \partial_m, \mathcal{A}_s^{\epsilon}]v&= Y^{l,m}_j \partial_l \partial_m (\psi D^\sigma \scb_{s-k}^{\epsilon} \eta u ) - \psi D^\sigma \scb_{s-k}^{\epsilon} \eta Y^{l,m}_j \partial_l \partial_m v\\
&= Y^{l,m}_j \partial_{lm}(\psi) D^\sigma \scb_{s-k}^{\epsilon} \eta v + 2(Y_j\psi) Y_j D^\sigma \scb_{s-l}^{\epsilon} \eta v \\
&\qquad + \psi Y^{l,m}_j \partial_{lm} D^\sigma \scb_{s-k}^{\epsilon} \eta v  - \psi D^\sigma \scb_{s-k}^{\epsilon} \eta Y^{l,m}_j \partial_l \partial_m v\\
&=(I)+(II)+(III) 
\end{align*}     
where $(I)$ and $(II)$ are the first two terms on the right-hand side of the last equality above.  We easily have the estimate
\begin{align*}
\|(I)\| \leq C \|\eta v \|_{(s)} < C.  
\end{align*}  
By Lemma \ref{lem:52}, we also have for some $\phi\in C_0^\infty(\Omega)$ 
\begin{align*}
\| (II)\| \leq C \|\phi Y_j D^\sigma \scb_{s-k}^\epsilon \eta u\| \leq C(\| Y_j \phi D^\sigma \scb_{s-k}^\epsilon \eta u\| +\|\eta u\|_{(s)})< C
\end{align*}
The estimate for $$(III)= \psi Y^{l,m}_j \partial_{lm} D^\sigma \scb_{s-k}^\epsilon \eta v  - \psi D^\sigma \scb_{s-k}^\epsilon \eta Y^{l,m}_j \partial_l \partial_m v=(III)'-(III)''$$ is more involved and we begin by further decomposing $(III)''$.  Realize by Taylor's formula we have
\begin{align*}
&(III)''(x) = \psi(x) D^\sigma \scb_{s-k}^\epsilon \eta Y^{l,m}_j \partial_{lm} v(x)\\
&= \psi(x) \sum_{0\leq |\tau |\leq k+1}(-1)^{|\tau|}(\tau !)^{-1}D^\tau Y^{l,m}_j(x) \big\langle ( x-\cdot)^\tau D^\sigma G_{k-s}^\epsilon (x- \cdot), \eta(\cdot) \partial_{lm}v(\cdot) \big\rangle \\
&\qquad + R(x).  
\end{align*}
Using the assumed regularity of $Y^{l,m}_j$ and Lemma 6.1, we can obtain the bound the desired bound for $R$:
\begin{align*}
\| R\| \leq C\| \psi v\|_{(s)} < C.   
\end{align*}
Moreover, any terms in the sum above with $|\tau |\geq 2$ also have the same estimate as $R$, so what remains to bound in $(III)''$ is 
\begin{align*}
&\psi(x) Y^{l,m}_j(x) \big\langle  D^\sigma G_{k-s}^\epsilon (x- \cdot), \eta(\cdot) \partial_{lm} v(\cdot) \big \rangle \\
&\qquad - \sum_{n=1}^d \psi(x) \partial_nY^{l,m}_j(x) \big\langle (x_n-\cdot) D^\sigma G_{k-s}^\epsilon(x-\cdot), \eta(\cdot) \partial_{lm} v(\cdot)   \big\rangle \\
&= (III)' + T(x) 
\end{align*}
where $T$ satisfies the following bound 
\begin{align*}
\|T\| \leq  C\bigg(\| \psi Y_j \widetilde{\scb}_{s}^{\epsilon} \eta v\|+ \sum_{m=1}^d \| \psi Y_j D^\sigma \scb_{s-k}^\epsilon \partial_m(\eta) v\| +\sum_{|\tau|\leq 2} \| (D^\tau\eta)v\|_{(s)}
\bigg)  
\end{align*}
for some $\widetilde{B}_s^\epsilon$ which is the sum of terms of the form required in Lemma \ref{lem:52}.  Doing as before with $(II)$, the result now follows by Lemma \ref{lem:52}.    
\end{proof}

\section{The commutator estimates}
\label{sec:Mcommest}

The goal of this section is to prove Lemma \ref{lem:tough_rcalc}.  Even though the argument establishing this result is long, the idea behind it is basic.  First, we will write out the operator $[V, \mcm ]$ explicitly, noting it can be decomposed as the sum of operators of a small number of distinct forms.  Using the Bessel kernel estimates of Section \ref{sec:frac}, we will then bound each of these quantities and their commutators with $V$ in the claimed fashion.  

Part of the novelty in the following computations is knowing when and how to integrate by parts so that minimal regularity on the vector field $V$ is enforced.     

So that the mathematical expressions of this section are compact, for $k\geq 1$ and $f,g,h: \mathbb{R}^d \rightarrow \mathbb{C}$, let $(f \Delta_g^k *h)(x)$ denote the function 
\begin{align*}          
\int_{\mathbb{R}^d} f(y) \bigg(g(x)-\sum_{|\sigma|\leq k-1 } \frac{D^\sigma g(y)}{\sigma !}(x-y)^\sigma \bigg)h(x-y) \, dy    
\end{align*}
whenever it is defined.  If $k=1$, as in Section \ref{sec:frac} we will simply write $\Delta_g$ instead of $\Delta_g^1$.

We now proceed as described above.  Fix $Q\subset \Omega$ compact and let $u\in C_0^{\infty}(Q: \mathbb{C})$ be arbitrary.  First observe that   
\begin{align*}
[V, \mcm ] u &= a( b\Delta_{V^l} \partial_i \partial_j \partial_l u *J) +  V^l \partial_l(a)( b \partial_i \partial_j  u *J )+ V^l a ( \partial_l(b) \partial_i \partial_j u * J)\\
\nonumber & - a(b \partial_j(V^l) \partial_i \partial_l u * J)- a(b \partial_i (V^l) \partial_j \partial_l u * J)- a(b \partial_i \partial_j(V^l) \partial_l u * J ).  
\end{align*}
Using this expression, leave terms that are of the same form as $\mcm$ as is and then integrate by parts on the remaining terms to find that 
\begin{align*}
[V, \mcm ] u &= a( b\Delta_{V^l} \partial_j \partial_l u *\partial_i J) - a ( \partial_i(b \Delta_{V^l}) \partial_j \partial_l u * J )+V^l \partial_l(a)( b \partial_i \partial_j  u * J)\\
\nonumber&+ V^l a ( \partial_l(b) \partial_i \partial_j u * J) - a(b \partial_j(V^l) \partial_i \partial_l u * J)- a(b \partial_i (V^l) \partial_j \partial_l u * J)\\
\nonumber &  - a(b \partial_i \partial_j(V^l) u * \partial_l J) + a( \partial_l(b \partial_i \partial_j(V^l)) u * J).
\end{align*}
For the first term on the right side of the previous equality, introduce an additional difference to see that 
\begin{align*}
[V, \mcm ] u &= a( b\Delta_{V^l}^2 \partial_j \partial_l u *\partial_i J) + a(b \partial_m(V^l)\partial_j \partial_l u *x^m\partial_i J) \\&- a ( \partial_i(b \Delta_{V^l}) \partial_j \partial_l u * J )
+ V^l \partial_l(a)( b \partial_i \partial_j  u * J) + V^l a ( \partial_l(b) \partial_i \partial_j u *J) \\&- a(b \partial_j(V^l) \partial_i \partial_l u * J) - a(b \partial_i (V^l) \partial_j \partial_l u * J) - a(b \partial_i \partial_j(V^l) u * \partial_l J)\\
& + a( \partial_l(b \partial_i \partial_j(V^l)) u * J). 
\end{align*}
Let $[V, \mcm]u= T_1u + T_2 u  $ where 
\begin{align}
\nonumber T_1 u &= a( b\Delta_{V^l}^2 \partial_j \partial_l u *\partial_i J) \\
\label{eqn:comm_dec_1}T_2 u & = a(b \partial_m(V^l)\partial_j \partial_l u *x^m\partial_i J) - a ( \partial_i(b \Delta_{V^l}) \partial_j \partial_l u * J )\\
&\nonumber 
+ V^l \partial_l(a)( b \partial_i \partial_j  u * J) + V^l a ( \partial_l(b) \partial_i \partial_j u *J) - a(b \partial_j(V^l) \partial_i \partial_l u * J)\\& \nonumber - a(b \partial_i (V^l) \partial_j \partial_l u * J) - a(b \partial_i \partial_j(V^l) u * \partial_l J) + a( \partial_l(b \partial_i \partial_j(V^l)) u * J). 
\end{align}
Since we are permitted two applications of integration by parts, note that 
\begin{align}
\label{eqn:comm_dec_2}T_1 u &= a(b \Delta^2_{V^l} u * \partial_l \partial_j \partial_i J) - a (\partial_{l}(b) \Delta^2_{V^l} u * \partial_j \partial_i J)\\ \nonumber &+ a( b \partial_{l}\partial_{m}(V^l) u * x^{m} \partial_j \partial_i J) -a (\partial_{j}(b) \Delta^2_{V^l} u * \partial_l \partial_i J)\\ \nonumber &+ a( b \partial_{j}\partial_{m}(V^l) u * x^{m} \partial_l  \partial_i J)
 + a( \partial_l \partial_j(b) \Delta^2_{V^l} u * \partial_i J) \\
 &\nonumber - a(b \partial_l \partial_m \partial_j(V^l) u * x^m \partial_i J)+ a(b \partial_l \partial_j(V^l) u *  \partial_i J)\\
\nonumber &\qquad - a( \partial_l(b) \partial_j \partial_m(V^l)u* x^m \partial_i J)- a( \partial_j(b) \partial_l \partial_m(V^l)u* x^m \partial_i J). 
\end{align}
Recycling the notation for $J$ above, observe that each term in $[V, \mcm ]u=T_1u+ T_2 u$ is of one of the following eight general forms
\begin{align*}
&A_1 u= a(b F u *  J), && A_2 u =a(b  F u * K)\\
& A_3 u = a(\Delta_{F}^2 b u * K), && A_4 u=a(\Delta_{F}^2 b u * \partial_i K)\\
& A_5 u = a(b \Delta_F^2 u * \partial_l \partial_i K), \qquad &&A_6 u = a (b F \partial_{i} \partial_j u * J)\\
& A_7  u= a F (b \partial_i \partial_j u * J), && A_8 u = a F(b G \partial_i \partial_j u * J).        
\end{align*}
where $J \in \mathcal{J}_\alpha$, $K \in \mathcal{J}_{\alpha-1}$ and $a,b, F, G \in C^{\infty}(\mathbb{R}^d: \mathbb{R})\cap B(\mathbb{R}^d:\mathbb{R})$.  We use the capital letters $F$ and $G$ to emphasize those terms which may depend on the coefficients of $V$.

\subsection{Bounding $\|A_i u\|_{(-\beta)}$}

Let $\beta >0$ be such that $\alpha + \beta >2$ and fix $\gamma \in (2-\alpha,1)$.  We now estimate $\|A_i u \|_{(-\beta)}$ for $i=1,2,\ldots, 8$ while keeping careful track of how the constants in the bounds depend on $F$ and $G$.  As in previous arguments, all constants below will depend on $Q$ but not on $u\in C_0^{\infty}(Q: \mathbb{R})$. 

Referring to Section \ref{sec:frac}, Lemma \ref{lem:kernel_a} is easily seen to imply the following bounds: 
\begin{align*}
\label{eqn:est1}\|A_1 u\|_{(-\beta)} &\leq C_1 \|\,\, |Fu|\,\, \|_{(-\alpha-\beta)}, \tag{b1}\\    
\|A_2 u \|_{(-\beta)} &\leq C_2 \| \,\, |Fu | \,\, \|_{(-\alpha -\beta +1)}.\tag{b2}
\end{align*}
where $C_1, C_2>0$ are independent of $F$ and $G$.  To bound $\| A_3 u \|_{(-\beta)}$ first notice 
\begin{align*}
|\scb_{-\beta} A_3 u | \leq 2 \|a\|_{\infty}\|b\|_{\infty}\big[\|F\|_{\infty} \scb_{-\beta} (|u|* |K|) + \scb_{-\beta} (|\partial_j(F) u|* |x^j K|)\big].  
\end{align*}
Again, applying Lemma \ref{lem:kernel_a} we find:
\begin{align*}
\tag{b3}
\| A_3 u \|_{(-\beta)} \leq C_3(\|F\|_{\infty}) \Big[\|u\|+ \| \, \, |F'u| \, \, \|_{(-\alpha-\beta)}\Big].
\end{align*}
Since 
\begin{align*}
|\scb_{-\beta} A_4 u | \leq \|a\|_{\infty} \|b\|_{\infty}[ |F|_{\delta}\scb_{-\beta}(|u|*|x|^\gamma |\partial_i K|) + \scb_{-\beta}(|\partial_j(F) u|* |x^j\partial_i K|)],
\end{align*}
we also obtain the bound:
\begin{align*}
\tag{b4}\|A_4 u \|_{(-\beta)}& \leq C_4(|F|_{\delta})  \Big[\|u \| + \|\, \, |F'u |\,\,  \|_{(-\alpha- \beta +1)}\Big].  
\end{align*}
In a similar fashion, it is not hard to see that 
\begin{align*}
\tag{b5}
\|A_5 u \|_{(-\beta)}& \leq C_5(|F|_{1+\delta}) \|u\|.
\end{align*}

For $\| A_6 u \|_{(-\beta)}$, integrate by parts to obtain  
\begin{align*}
\scb_{-\beta} A_6 u &= \scb_{-\beta}a(b F \partial_j u * \partial_i J) - \scb_{-\beta} a(\partial_i(b F) \partial_j u*   J)\\
&= \scb_{-\beta}a(b F\partial_j u * \partial_i J) - \scb_{-\beta} a(\partial_i(b F)  u*  \partial_j J) + \scb_{-\beta} a(\partial_j \partial_i(b F)  u*  J)\\
&= \scb_{-\beta}a\partial_j (b F u * \partial_i J)-  \scb_{-\beta}a(\partial_j (b F) u * \partial_i J) - \scb_{-\beta} a(\partial_i(b F)  u*  \partial_j J)\\
&\qquad  +\,  \scb_{-\beta} a(\partial_j \partial_i(b F)  u*  J ).
\end{align*}
Let 
\begin{align*}
T_1 u &= -  \scb_{-\beta}a(\partial_j (b F) u * \partial_i J) - \scb_{-\beta} a(\partial_i(b F)  u*  \partial_j J)\\
&\qquad  +\,  \scb_{-\beta} a(\partial_j \partial_i(b F)  u*  J),\\
T_2 u &=\scb_{-\beta}a\partial_j (b F u * \partial_i J).
\end{align*}
It is plain that 
\begin{align*}
\|T_1 u \| \leq C\Big[\| \,\,|Fu|\,\, \|_{(-\alpha-\beta+1)} + \|\,\, |F' u| \,\, \|_{(-\alpha -\beta +1)} + \| \,\, |F'' u|\, \, \|_{(-\alpha -\beta)}\Big]
\end{align*}
where the constant $C>0$ is independent of $F$.  
To bound $\|T_2 u\|$, write
\begin{align*}
T_2 u &=\scb_{-\beta}a\partial_j (b F u * \partial_i J)\\
&= a \scb_{-\beta}\partial_j (b F u * \partial_i J)- \scb_{-\beta}\Delta_a\partial_j (b F u * \partial_i J)\\
& = a \scb_{-\beta}\partial_j (b F u * \partial_i J)- \scb_{-\beta}\partial_j(a) (b F u * \partial_i J) + \Delta_a ( b F u * \partial_i J) * \partial_j G_\beta
\end{align*}
where the last equality follows by integration by parts.  From this, we see that 
\begin{align*}
\|T_1 u \| \leq C\|\,\, |Fu|\,\, \|_{(-\alpha -\beta +2)}.
\end{align*} 
where $C>0$ is independent of $F$.  Putting the bounds for $\|T_1 u\|$ and $\|T_2 u\|$ together we see that
\begin{align*}
\| A_6 u \|_{(-\beta)}\leq C_6 \Big[\|\, \, |Fu| \,\, \|_{(-\alpha -\beta +2)}+ \| \,\, |F' u |\, \, \|_{(-\alpha -\beta +1)} + \|\, \, |F'' u |\,\, \|_{(-\alpha-\beta)}\Big] \tag{b6}
\end{align*}
where $C_6 >0$ is independent of $F$.  Using the very same process as in the estimate for $\| A_6 u\|_{(-\beta)}$, we also obtain the following bounds:
\begin{align*}
\|A_7 u \|_{(-\beta)} &\leq C_7(|F|_{1}) \|u\|,\tag{b7} \\
\label{eqn:est8}\|A_8 u \|_{(-\beta)}&\leq  C_8(|F|_{1})[\| \,\, |Gu| \,\, \|_{(-\alpha -\beta +2)}+ \| \,\, |G' u | \, \,\|_{(-\alpha -\beta +1)} + \| \,\, |G'' u |\,\, \|_{(-\alpha-\beta)}], \tag{b8}
\end{align*}

\subsection{Bounding $\|[G\partial_k, A_i]u \|_{(-\beta)}$}  For the proof of Lemma \ref{lem:tough_rcalc}, we will only need to estimate $\|[G\partial_k, A_i]u \|_{(-\beta)}$, $i=1,2,\ldots, 7$.

For $\| [G\partial_k, A_1]u\|_{(-\beta)}$, write 
\begin{align*}
\scb_{-\beta} [G \partial_k , A_1]u &= \scb_{-\beta} \partial_k(a) G(b F u * J)+ \scb_{-\beta} a G(b F u * \partial_k J)\\
&\qquad  -\scb_{-\beta}  a(b F G \partial_k u * J)\\
&= \scb_{-\beta} \partial_k(a) G(b Fu * J)+ \scb_{-\beta}  a G(b F u * \partial_k J) \\
&\qquad - \scb_{-\beta} a(b F G  u *\partial_k J )+ \scb_{-\beta} a( \partial_k(b  F G ) u *J).  
\end{align*}
Then the estimate 
\begin{align*}
\| [G \partial_k, A_1]u \|_{(-\beta)}\leq D_1 (|G|_{1}) \Big[\, \| \,\, |Fu|\, \,\|_{(-\alpha -\beta +1)}+ \|\, \, |F'u| \,\, \|_{(-\alpha-\beta)}\, \Big], \tag{b9}
\end{align*}
follows immediately.

For $\| [G \partial_k, A_2 ]u\|$, note that 
\begin{align*}
\scb_{-\beta} [G \partial_k , A_2]u &= \scb_{-\beta} \partial_k(a) G(b F u * K)+ \scb_{-\beta} a G(\partial_k(b F u) * K) \\
&\qquad - \scb_{-\beta} a(b F  G \partial_k u * K)\\
&=  \scb_{-\beta} \partial_k(a) G(b F u * K)+ \scb_{-\beta} a( \Delta_{G} b F \partial_k  u* K)\\
&\qquad  + \scb_{-\beta} aG(\partial_k(b F) u * K)\\
&=\scb_{-\beta} \partial_k(a) G(b F u * K)+ \scb_{-\beta}  a( \Delta_{G} b F u*\partial_k K) \\
&\qquad -\scb_{-\beta}  a(\partial_k ( \Delta_{G} b F)  u* K )+\, \scb_{-\beta}  aG(\partial_k(b F) u *K).  
\end{align*}
The bound  
\begin{align*}
\| [G \partial_k, A_2]u \|_{(-\beta)}\leq D_2(|G|_{1}) \Big[\,\| \, \,|Fu|\,\, \|_{(-\alpha -\beta +1)}+ \| \, \, |F'u| \,\, \|_{(-\alpha-\beta+1)}\,\Big] \tag{b10}
\end{align*}
thus follows.

Now for $\| [G \partial_k, A_3]u\|_{(-\beta)}$, write
\begin{align*}
\scb_{-\beta} [G \partial_k, A_3]u & = \scb_{-\beta} G \partial_k(a) ( \Delta_F^2 b u * K) + \scb_{-\beta} G a ( \Delta_{\partial_k(F)} b u * K)  \\
& + \scb_{-\beta} G a ( \Delta_F^2 b u * \partial_k K)  - \scb_{-\beta} a( \Delta_F^2 bG \partial_k u * K) \\
&= \scb_{-\beta} G \partial_k(a) ( \Delta_F^2 b u * K) + \scb_{-\beta} G a ( \Delta_{\partial_k(F)} b u * K)  \\
&+ \scb_{-\beta} G a (  \partial_k(\Delta_F^2 b) u * K)  + \scb_{-\beta} a( \Delta_F^2 \Delta_G b u *\partial_k  K).  
\end{align*}
Using this expression, it is not hard to obtain the estimate:  
\begin{align*}
\| [G \partial_k, A_3]u \|_{(-\beta)}&\leq D_3( |F|_{1}, \|G\|_{\infty})\Big[ \|u\|+ \| \, |F''u | \,\|_{(-\alpha-\beta)}\Big] \tag{b11}
\end{align*}

For  $\|[G \partial_k, A_4]u\|_{(-\beta)}$, notice
\begin{align*}
\scb_{-\beta} [G \partial_k, A_4]u &= \scb_{-\beta} G \partial_k(a) ( \Delta_F^2 b u * \partial_i K) + \scb_{-\beta} Ga ( \Delta_{\partial_k(F)} b u * \partial_i K)\\
&+  \scb_{-\beta} G a ( \Delta_F^2 b u *\partial_k \partial_i K)-  \scb_{-\beta} a ( \Delta_F^2 b G \partial_k u * \partial_i K) \\
&=  \scb_{-\beta} G \partial_k(a) ( \Delta_F^2 b u * \partial_i K) + \scb_{-\beta} Ga ( \Delta_{\partial_k(F)} b u * \partial_i K)\\
&+  \scb_{-\beta} a ( \partial_k(\Delta_F^2 b G) u * \partial_i K)  +  \scb_{-\beta} a( \Delta_F^2 \Delta_G b u * \partial_k \partial_i K).   
\end{align*}
From this we obtain the bound
\begin{align*}
\| [G \partial_k, A_4]u \|_{(-\beta)}\leq D_4(|F|_{ 1+ \delta}, |G|_{1})\Big[ \|u\|+ \|\, \, |F''u| \,\, \|_{(-\alpha-\beta+1)}\Big]. \tag{b12}
\end{align*}

  Turning to the estimate for $\|[G \partial_k, A_5]u \|_{(-\beta)}$, write 
\begin{align*}
\scb_{-\beta} [G \partial_k, A_5]u &= \scb_{-\beta} G \partial_k(a) ( \Delta_F^2 b u * \partial_j\partial_i K) + \scb_{-\beta} a G\partial_k ( \Delta_F^2 b u * \partial_j \partial_{i} K) \\
& \qquad - \scb_{-\beta} a ( \Delta_F^2 G b  \partial_k u * \partial_j \partial_{i} K)\\
&=\scb_{-\beta} G \partial_k(a) ( \Delta_F^2 b u * \partial_j\partial_i K) + \scb_{-\beta} a G \partial_k ( \Delta_F^3 b u * \partial_j \partial_i K) \\
&- \scb_{-\beta} a ( \Delta_F^3 G b  \partial_k u * \partial_j \partial_{i} K) + \scb_{-\beta} a G \partial_k( \partial_l \partial_m(F) b u * x^m x^l \partial_j \partial_i K) \\
&-\scb_{-\beta} a ( \partial_l \partial_m(F) G b \partial_k u * x^m x^l \partial_j \partial_i K). 
\end{align*}
Unraveling the previous expression further we obtain
\begin{align*}
\scb_{-\beta} [G \partial_k, A_5]u &=\scb_{-\beta} G \partial_k(a) ( \Delta_F^2 b u * \partial_j\partial_i K) + \scb_{-\beta} a G  ( \Delta_{\partial_k(F)}^2 b u * \partial_j \partial_i K) \\
&+ \scb_{-\beta} a G  ( \Delta_{F}^3 b u * \partial_k \partial_j \partial_i K)+ \scb_{-\beta} a (  \partial_k(\Delta_F^3 G b ) u * \partial_j \partial_{i} K) \\
&- \scb_{-\beta} a ( \Delta_F^3 G bu *   \partial_k \partial_j \partial_{i} K) + \scb_{-\beta} a G \partial_k( \partial_l \partial_m(F) b u * x^m x^l \partial_j \partial_i K) \\
&-\scb_{-\beta} a\partial_k ( \partial_l \partial_m(F) G b  u * x^m x^l \partial_j \partial_i K)\\
& + \scb_{-\beta} a(\partial_k ( \partial_l \partial_m(F) G b)  u * x^m x^l \partial_j \partial_i K).  
\end{align*}
From this expression we may deduce the bound 
\begin{align*}
\label{comm:5}
&\|[G \partial_k, A_5]u  \|_{(-\beta)} \leq C(|F|_{2+ \delta}, |G|_{1})\Big[\|u\|+ \| \,\, |F''' u | \, \,\|_{(-\alpha -\beta +1)}\Big]. \tag{b13}
\end{align*}
For $\| [G\partial_k,A_6]u\|_{(-\beta)}$, use \eqref{eqn:comm_dec_1} and \eqref{eqn:comm_dec_2} replacing $V^l$, $\partial_l$, and $b$ with $G$, $\partial_k$, and $b F$ respectively.  Then apply the first line of estimates \eqref{eqn:est1}-\eqref{eqn:est8} for each of these terms to obtain the estimate:
\begin{align*}
\tag{b14}\| [G \partial_k, A_6]u \|_{(-\beta)} &\leq D_6(|F|_{1}, |G|_{1+ \delta})\Big[\, \|u\| + \| \,\, |F'' u |\, \, \|_{(-\alpha -\beta +1)}\\
\nonumber &\qquad \qquad + \| \,\, |F''' u | \, \,\|_{(-\alpha -\beta)}+ \|\, \, |G'' u | \, \,\|_{(-\alpha -\beta+1)} \, +\| \,\, |G''' u | \,\, \|_{(-\alpha -\beta)}\Big]
\end{align*}
One can similarly estimate $\| [G\partial_k,A_7]u\|_{(-\beta)}$ to see that 
\begin{align*}
\tag{b15}&\| [G \partial_k, A_7]u \|_{(-\beta)} \leq D_6(|F|_{2}, |G|_{1+ \delta})\big[ \,\|u \| + \| \,\, |G''u|\, \, \|_{(-\alpha -\beta +1)}\\
\nonumber &\qquad \qquad\qquad\qquad  \qquad\qquad\qquad \qquad\qquad + \| \,\, |G'''u|\, \, \|_{(-\alpha -\beta +1)}\, \big]
\end{align*}

We now finish the proof of Lemma \ref{lem:tough_rcalc}.  The bound for $\| [V, \mathcal{M}]u\|_{(-\beta)}$ follows from \eqref{eqn:comm_dec_1} and \eqref{eqn:comm_dec_2} and the estimates  \eqref{eqn:est1}-\eqref{eqn:est8} for $\|A_i u \|_{(-\beta)}$, $i=1,2,\ldots, 8$. The estimate for $\| [V,[V, \mathcal{M}]]u \|_{(-\beta)}$ follows from \eqref{eqn:comm_dec_1} and \eqref{eqn:comm_dec_2} and the bounds (b9)-(b15).

\bibliographystyle{plain}
\bibliography{Rbiblio}

\end{document}